\documentclass[12pt]{amsart}

\usepackage[colorlinks=true, urlcolor=blue,bookmarks=true,bookmarksopen=true,
citecolor=blue
]{hyperref}
\usepackage{graphicx}

\usepackage{epsfig}
\usepackage{amscd}
\usepackage[mathscr]{eucal}
\usepackage{amssymb}
\usepackage{amsxtra}
\usepackage{amsmath}
\usepackage[all]{xy}
\usepackage{enumerate}
\usepackage{mathrsfs}

\usepackage{color}
\theoremstyle{plain}
\newtheorem{mainthm}{Theorem}
\newtheorem{maincor}{Corollary}
\newtheorem{thm}{Theorem}[subsection]
\newtheorem{cor}[thm]{Corollary}
\newtheorem{lem}[thm]{Lemma}

\theoremstyle{definition}
\newtheorem{dfn}[thm]{Definition}

\theoremstyle{remark}
\newtheorem{rem}[thm]{Remark}

\newtheorem*{exnonum}{Example}
\newtheorem*{exsnonum}{Examples}

\newcommand{\cntrs}{\setcounter{thm}{0}
  \renewcommand{\thethm}{\thesection.\Alph{thm}}}
\newcommand{\cntrsb}{\setcounter{thm}{0}
  \renewcommand{\thethm}{\thesubsection.\Alph{thm}}}

\newcommand{\Qed}{\hfill \qedsymbol \medskip}

\newcommand{\tdef}{{\textnormal{def}}}
\newcommand{\hooklongrightarrow}{\lhook\joinrel\longrightarrow}

\begin{document}

\title{The symplectic topology of projective manifolds with small dual}

\author{Paul Biran} \author{Yochay Jerby}

\date{\today}

\address{Paul Biran, Departement Mathematik, ETH, 8092 Z\"{u}rich, Switzerland}
\email{biran@math.ethz.ch}

\address{Yochay Jerby, Departement Mathematik, ETH, 8092 Z\"{u}rich, Switzerland}
 \email{yochay.jerby@math.ethz.ch}

\bibliographystyle{alphanum}

%
%
% ...

\maketitle

%
%

\begin{abstract}
   We study smooth projective varieties with small dual variety using
   methods from symplectic topology. For such varieties we prove that
   the hyperplane class is an invertible element in the quantum
   cohomology of their hyperplane sections. We also prove that the
   affine part of such varieties are subcritical. We derive several
   topological and algebraic geometric consequences from that. The
   main tool in our work is the Seidel representation associated to
   Hamiltonian fibrations.
\end{abstract}

\section{introduction and summary of the main results} \cntrs
\label{S:Intro}

In this paper we study a special class of complex algebraic manifolds
called {\em projective manifolds with small dual}.  A projectively
embedded algebraic manifold $X \subset \mathbb{C}P^N$ is said to have
small dual if the dual variety $X^{\ast} \subset
(\mathbb{C}P^N)^{\ast}$ has (complex) codimension $\geq 2$.  Recall
that the dual variety $X^*$ of a projectively embedded algebraic
manifold $X \subset {\mathbb{C}}P^N$ is by definition the space of all
hyperplanes $H \subset {\mathbb{C}}P^N$ that are not transverse to
$X$, i.e.
$$X^* = \{H \in ({\mathbb{C}}P^N)^* \mid H 
\textnormal{ is somewhere tangent to } X\}.$$ Let us mention that for
``most'' manifolds the codimension of $X^*$ is $1$, however in special
situations the codimension might be larger. To measure to which extent
$X$ deviates from the typical case one defines the
defect~\footnote{Some authors call this quantity {\em dual defect} to
  distinguish it from other ``defects'' appearing in projective
  geometry, such as {\em secant defect}, see
  e.g.~\cite{Landsberg:tensors, Zak:tangents-secants}. In this paper
  we will however stick to the wording ``defect'', which is attributed
  in~\cite{Ein:dual1, Ein:dual2} to A.~Landman.} of an algebraic
manifold $X \subset \mathbb{C}P^N$ by
$$\tdef(X) = \textnormal{codim}_{\mathbb{C}}(X^{\ast}) -1.$$ 
Thus we will call manifolds with small dual also {\em manifolds with
  positive defect}. Note that this is not an intrinsic property of
$X$, but rather of a given projective embedding of $X$.

The class of algebraic manifolds with small dual was studied by many
authors, for instance see~\cite{Ein:dual1, Ein:dual2,
  Gr-Ha:alg-geom-loc-diff, Kl:tangency, Sn:nef-val,
  Belt-Som:adjunction}, see also~\cite{Te:dual} for a nice survey.
The study of the relation between $X^*$ and the topology of $X$ (and
its hyperplane sections) had been initiated earlier
in~\cite{An-Fr:Lefschetz-2}.  These works show that manifolds with
small dual have very special geometry. In this paper we will show that
such manifolds also exhibit unique properties from the point of view
of symplectic topology.
 
Our main results are concerned with geometric properties of a smooth
hyperplane section $\Sigma \subset X$ of a manifold $X \subset
\mathbb{C}P^N$ with small dual, under the additional assumption that
$b_2(X)=1$. (Here and in what follows we denote by $b_j(X) = \dim
H^j(X;\mathbb{R})$ the $j$'th Betti-number of $X$.) By a well known
result of Ein~\cite{Ein:dual1} the assumption $b_2(X)=1$ implies that
both $X$ and $\Sigma$ are Fano manifolds.

For a space $Y$ we will denote from now on by $$H^*(Y) :=
H^*(Y;\mathbb{Z})/\textnormal{torsion}$$ the torsion-free part of the
integral cohomology $H^*(Y;\mathbb{Z})$. Denote by
$QH^*(\Sigma;\Lambda) = (H^{\bullet}(\Sigma) \otimes \Lambda)^*$ the
quantum cohomology ring of $\Sigma$ with coefficients in the Novikov
ring $\Lambda=\mathbb{Z}[q,q^{-1}]$ (see below for our grading
conventions), and endowed with the quantum product $*$. We prove:

\begin{mainthm} \label{mt:om-inv} Let $X \subset \mathbb{C}P^N$ be an
   algebraic manifold with small dual, $b_2(X)=1$ and
   $\dim_{\mathbb{C}}(X) \geq 2$. Let $\Sigma$ be a smooth hyperplane
   section of $X$. Let $\omega$ be the restriction of the Fubini-Study
   K\"{a}hler form of $\mathbb{C}P^N$ to $\Sigma$. Then
   $$[\omega] \in QH^2(\Sigma ; \Lambda)$$ is an invertible
   element with respect to the quantum product.
\end{mainthm}
We will actually prove a slightly stronger result
in~\S\ref{s:small-dual-fibr} (see Theorem~\ref{t:om-inv-gnrl} and the
discussion after it). In Theorem~\ref{t:gnrl-S-pi-ell}
in~\S\ref{s:non-monotone} we will establish a much more general,
though less precise, version of this theorem.

A classical result of Lanteri and Struppa~\cite{La-St:top} (see
also~\cite{An-Fr:Lefschetz-2}) on the topology of projective manifolds
with positive defect states that if $X \subset \mathbb{C}P^N$ is a
projective manifold with $dim_{\mathbb{C}} X = n$ and $\tdef(X)=k>0 $
then:
$$b_j(X)=b_{j+2}(X) \; \; \forall \, n-(k-1) \leq j \leq n+k-1.$$
(In~\S\ref{S:subcri-def} we will reprove this fact using Morse
theory). As we will see in Corollary~\ref{mc:periodicity} below,
Theorem~\ref{mt:om-inv} implies stronger topological restrictions in
the case $b_2(X)=1$.

As mentioned above, under the assumption $b_2(X)=1$ the manifold
$\Sigma$ is Fano. The quantum cohomology $QH^{\ast}(\Sigma ; \Lambda)
= (H^{\bullet}(\Sigma) \otimes \Lambda)^*$ admits a grading induced
from both factors $H^{\bullet}(\Sigma)$ and $\Lambda$. Here we grade
$\Lambda$ by taking $deg(q) = 2C_{\Sigma} $, where
$$C_{\Sigma} = \min \left \{ c_1^{\Sigma}(A) >0 
   \mid A \in \textnormal{image\,} (\pi_2(\Sigma) \rightarrow
   H_2(\Sigma ; \mathbb{Z})) \right \} \in \mathbb{N}$$ is the minimal
Chern number of $\Sigma$. Here we have denoted by $c_1^{\Sigma} \in
H^2(\Sigma;\mathbb{Z})$ the first Chern class of the tangent bundle
$T\Sigma$ of $\Sigma$.  Theorem~\ref{mt:om-inv} implies that the map
$$\ast [\omega] : QH^{\ast} (\Sigma ; \Lambda) \longrightarrow
QH^{\ast+2}(\Sigma ; \Lambda), \quad a \longmapsto a*[\omega],$$ is an
isomorphism. In our case, a computation of Ein~\cite{Ein:dual1} gives:
$$2C_X = n+k+2, \quad 2C_{\Sigma}=n+k.$$ (It is well known, by a
result of Landman, that $n$ and $k$ must have the same parity.
See~\S\ref{S:PDS}.)  Define now the cohomology of $X$ graded
cyclically as follows:
\begin{equation} \label{eq:chmlg-cyc-grd} \widetilde{H}^i(X) =
   \bigoplus_{l \in \mathbb{Z}} H^{i+2C_X l}(X), \quad
   \widetilde{b}_i(X) = \textnormal{rank}\, \widetilde{H}^i(X).
\end{equation}
Define $\widetilde{H}^i(\Sigma)$ and $\widetilde{b}_i(\Sigma)$ in a
similar way (note that in the definition of $\widetilde{H}^i(\Sigma)$
one has to replace also $C_X$ by $C_{\Sigma}$).
Theorem~\ref{mt:om-inv} together with a simple application of the
Lefschetz hyperplane section theorem give the following result:

\addtocounter{mainthm}{1}
\begin{maincor} \label{mc:periodicity} Let $X \subset \mathbb{C}P^N$
   be an algebraic manifold with small dual and $b_2(X)=1$. Then
   $\widetilde{b}_j(X) = \widetilde{b}_{j+2}(X), \;\; \forall \, \, j
   \in \mathbb{Z}.$ Moreover, if $\Sigma \subset X$ is a smooth
   hyperplane section then similarly to $X$ we have
   $\widetilde{b}_j(\Sigma) = \widetilde{b}_{j+2}(\Sigma), \;\;
   \forall \, \, j \in \mathbb{Z}.$
\end{maincor}
A similar result (for subcritical manifolds) has been previously
obtained by He~\cite{He:subcrit} using methods of contact homology.

If $dim_{\mathbb{C}}(X) = n$ and $\tdef(X) = k$,
Theorem~\ref{mc:periodicity} implies the following relations among the
Betti numbers of $X$:
\begin{align*}
   & b_j(X) + b_{j+n+k+2}(X) = b_{j+2}(X) + b_{j+n+k+4}(X), \;\;
   \forall \, \, 0 \leq j \leq n+k-1,  \\
   & b_{n+k}(X) = b_{n+k+2}(X) + 1, \quad b_{n+k+1}(X) =
   b_{n+k+3}(X),
\end{align*}
and the following ones for those of $\Sigma$:
\begin{align*}
   & b_j(\Sigma) + b_{j+n+k}(\Sigma) = b_{j+2}(\Sigma) +
   b_{j+n+k+2}(\Sigma), \;\;
   \forall \, \, 0 \leq j \leq n+k-3,  \\
   & b_{n+k-2}(\Sigma) = b_{n+k}(\Sigma) + 1, \quad b_{n+k-1}(\Sigma) =
   b_{n+k+1}(\Sigma).
\end{align*}
We will prove a slightly stronger result in~\S\ref{s:small-dual-fibr},
see Corollary~\ref{c:periodicity-gnrl}.




\begin{exnonum}
   Consider the complex Grassmannian $X = Gr(5,2) \subset
   \mathbb{C}P^{9}$ of $2$-dimensional subspaces in $\mathbb{C}^5$
   embedded in projective space by the Pl\"{u}cker embedding.  It is
   known that $\tdef(X)=2$, see~\cite{Mumford:footnotes,
     Gr-Ha:alg-geom-loc-diff, Te:dual}. We have $\dim_{\mathbb{C}}(X)
   = 6$ and $2C_X = 10$. The table of Betti numbers of $X$ is given as
   follows:
   \begin{center}
      \begin{tabular}
         {|c|c|c|c|c|c|c|c|c|c|c|c|c|c|} 
         \hline q & 0 & 1 & 2 & 3  & 4 & 5 & 6 & 7 & 8 & 9 & 10 & 11 & 12 
         \\ \hline 
         $b_q(X)$ & 1 & 0 & 1 & 0 & 2 & 0 & 2 & 0 & 2 & 0 & 1 & 0 & 1 
         \\ \hline 
      \end{tabular}
   \end{center}
\end{exnonum}

Further implications of Theorem~\ref{mt:om-inv} are obtained by
studying the algebraic properties of the inverse $[\omega]^{-1}$.
First note that due to degree reasons the inverse element should be of
the form $$ [\omega]^{-1} = \alpha_{n+k-2} \otimes q^{-1} \in
QH^{-2}(\Sigma ; \Lambda)$$ where $ \alpha_{n+k-2} \in
H^{n+k-2}(\Sigma)$ is a nontrivial element. Moreover, this element
needs to satisfy the following conditions:
$$[\omega] \cup \alpha_{n+k-2} =0, \quad ([\omega] \ast
\alpha_{n+k-2})_1= 1,$$ where $([\omega ] \ast \alpha_{n+k-2})_1 \in
H^0(\Sigma)$ is determined by the condition that
$$\langle ([\omega] \ast \alpha_{n+k-2})_1 , - \rangle = 
GW_1^{\Sigma}(PD[\omega], PD(\alpha_{n+k-2}), -).$$ Here $PD$ stands
for Poincar\'{e} duality, and for $a \in QH^l(\Sigma;\Lambda)$ and $i
\in \mathbb{Z}$ we denote by $(a)_i \in H^{l-2i C_{\Sigma}}(\Sigma)$
the coefficient of $q^i$ in $a$.  The notation $GW^{\Sigma}_j(A,B,C)$
stands for the Gromov-Witten invariant counting the number of rational
curves $u:\mathbb{C}P^1 \to \Sigma$ passing through three cycles
representing the homology classes $A, B, C$ with
$c_1(u_*[\mathbb{C}P^1]) = j C_{\Sigma}$.

So in our case, the fact that $([\omega]*\alpha_{n+k-2})_1 \neq 0$
implies that $\Sigma$ is uniruled. The uniruldness of $\Sigma$ (as
well as that of $X$) was previously known and the variety of rational
curves on it was studied by Ein in~\cite{Ein:dual1}. Finally, note
that the uniruldness of $X$ follows also from the results of
He~\cite{He:subcrit} in combination with Theorem~\ref{mt:subcrit}
above.

The method of proof of Theorem~\ref{mt:om-inv} is an application of
the theory of Hamiltonian fibrations and, in particular, their Seidel
elements, see~\cite{Se:pi1}. In~\cite{Se:pi1} Seidel constructed a
representation of $\pi_1(Ham(\Sigma, \omega))$ on $QH(\Sigma ;
\Lambda)$ given by a group homomorphism $$S : \pi_1(Ham(\Sigma,
\omega)) \longrightarrow QH(\Sigma ; \Lambda)^{\times},$$ where
$QH(\Sigma ; \Lambda)^{\times}$ is the group of invertible elements of
the quantum cohomology algebra.

Theorem~\ref{mt:om-inv} follows from:
\begin{mainthm} \label{mt:ham-loop} Let $X \subset \mathbb{C}P^N$ be
   an algebraic manifold with small dual and $b_2(X)=1$. Let $\Sigma
   \subset X$ be a smooth hyperplane section of $X$ and denote by
   $\omega$ the symplectic structure induced on $\Sigma$ from
   ${\mathbb{C}}P^N$. There exists a {\em nontrivial} element $1 \neq
   \lambda \in \pi_1(Ham(\Sigma , \omega))$ whose Seidel element is
   given by
   $$S(\lambda) = [\omega] \in QH^2(\Sigma ; \Lambda).$$
\end{mainthm}
See Theorems~\ref{t:om-inv-gnrl} and~\ref{t:gnrl-S-pi-ell} for more
general statements.

Before we turn to examples, let us mention that by results
of~\cite{Be-Fa-So:discr}, based on Mori theory, the classification of
manifolds with small dual is reduced to the case $b_2(X)=1$. Here is a
list of examples of manifolds with small dual and $b_2(X)=1$
(see~\cite{Te:dual} for more details):
\begin{exsnonum}
   \begin{enumerate}
     \item $X= \mathbb{C}P^n \subset \mathbb{C}P^{n+1}$ has $
      \tdef(X)=n$. \label{ex-i:cpn}
     \item $X=Gr(2l+1,2)$ embedded via the Pl\"{u}cker embedding has
      $\tdef(X)=2$. (See~\cite{Mumford:footnotes,
        Gr-Ha:alg-geom-loc-diff, Te:dual}.)
     \item $X=\mathbb{S}_{5} \subset \mathbb{C}P^{15}$ the
      $10$--dimensional spinor variety has $\tdef(X)=4$.
      (See~\cite{Laz-VdVen:topics-geom-proj, Te:dual}).
      \label{ex-i:spinor}
     \item In any of the examples (1)--(3) one can take iterated
      hyperplane sections and still get manifolds with
      $\textnormal{def} > 0$ and $b_2=1$, provided that the number of
      iterations does not exceed the defect$-1$. (See~\S\ref{S:PDS}.)
      \label{ex-i:hyp-sect}
   \end{enumerate}
\end{exsnonum}

The manifolds in~(\ref{ex-i:cpn})--(\ref{ex-i:spinor}) together with
the corresponding hyperplane sections~(\ref{ex-i:hyp-sect}) are the
only known examples of projective manifolds with small dual and
$b_2(X)=1$, see~\cite{Belt-Som:adjunction, Sn:nef-val}.  On the basis
of these examples, it is conjectured in~\cite{Belt-Som:adjunction}
that all non-linear algebraic manifolds with $b_2(X)=1$ have
$\tdef(X)\leq 4$.

\subsubsection*{Organization of the paper}
The rest of the paper is organized as follows: In~\S\ref{S:PDS} we
recall basic facts on projective manifolds with small dual.
In~\S\ref{S:ham-fib} we review relevant results from the theory of
Hamiltonian fibrations and the Seidel representation.
In~\S\ref{s:small-dual-fibr} we explain the relation between manifolds
with small dual and Hamiltonian fibrations. In~\S\ref{s:proof-om-inv}
we prove Theorems~\ref{mt:om-inv} and~\ref{mt:ham-loop}.
In~\S\ref{S:subcri-def} we discuss the relation between manifolds with
small dual and subcritical Stein manifolds and derive some topological
consequences from that. Corollary~\ref{mc:periodicity} is proved
in~\S\ref{s:prf-periodicity-gnrl}. In~\S\ref{s:further} we present
more applications of our methods to questions on the symplectic
topology and algebraic geometry of manifolds with small dual.  We also
outline an alternative proof of Corollary~\ref{mc:periodicity} based
on Lagrangian Floer theory. In~\S\ref{s:non-monotone} we explain how
to generalize Theorem~\ref{mt:om-inv} to the case $b_2(X)>1$ (or more
generally to non-monotone manifolds). In the same section we also work
out explicitly such an example. Finally, in~\S\ref{S:discussion} we
discuss some open questions and further possible directions of study.

\subsubsection*{Acknowledgments} We would like to thank the referee
for pointing out to us the reference to the (second) paper of
Andreotti-Frankel~\cite{An-Fr:Lefschetz-2} and for useful remarks
helping to improve the quality of the exposition.

\section{Basic results on projective manifolds with small dual} \cntrs
\label{S:PDS}

Let $X \subset \mathbb{C}P^N$ be an algebraic manifold of
$dim_{\mathbb{C}} X =n$. Denote by $(\mathbb{C}P^N)^{\ast}$ the dual
projective space parametrizing hyperplanes $H \subset \mathbb{C}P^N$.
To $X$ one associates the dual variety $X^{\ast} \subset
(\mathbb{C}P^N)^{\ast}$, which (in the case $X$ is smooth) is defined
as $$ X^{\ast} = \{ H \mid H \textnormal{ is somewhere tangent to }
X\}.$$ We refer the reader to~\cite{Te:dual} for a detailed account on
the subject of projective duality. In this section we will review
basic properties of projective manifolds with positive defect. Define
the defect of $X$ to be $$ \tdef(X)=\textnormal{codim}_{\mathbb{C}}
X^{\ast}-1.$$ Note that when $X^{\ast}$ is a hypersurface the defect
of $X$ is zero. An important feature of the defect is the following:
if $\tdef(X)=k$ then for a smooth point of the dual variety, $H \in
X_{\textnormal{sm}}^{\ast}$, the singular part $\textnormal{sing}(X
\cap H) $ of $X \cap H$ is a projective space of dimension $k$
linearly embedded in $\mathbb{C}P^N$. Thus, $X$ is covered by
projective spaces of dimension $k$, and in particular there is a
projective line through every point of $X$ (see~\cite{Kl:conormal}).

Next, the defect of $X$ and that of a hyperplane section $\Sigma
\subset X$ of $X$ are related as follows (see~\cite{Ein:dual2}):
\begin{equation} \label{eq:def-Sigma-X}
   \tdef(\Sigma) = \max \left\{ \tdef(X)-1, 0 \right\}.
\end{equation}
A well known (unpublished) result of Landman states that for manifolds
$X$ with small dual we have the following congruence
$\dim_{\mathbb{C}}(X) \equiv \tdef(X) \pmod{2}$ (see~\cite{Ein:dual1,
  Te:dual} for a proof of this).

Later, Ein proved in~\cite{Ein:dual1} the following.  {\sl Let $X
  \subset \mathbb{C}P^N$ be an algebraic manifold with
  $\dim_{\mathbb{C}}(X)=n$ and $\tdef(X)=k>0$. Denote by $c_1^X$ the
  first Chern class of $X$. Then through every point in $X$ there
  exists a projective line $S$ with
  $$c_1^X(S) = \frac{n+k}{2}+1.$$}

Of special importance is the case $b_2(X)=1$, which was extensively
studied by Ein in~\cite{Ein:dual1}. In this case we have:
\begin{equation}
   \label{eq:chern}
   c_1^X = \left(\frac{n+k}{2} +1 \right) \cdot h,
\end{equation}
where $h \in H^2(X) \cong \mathbb{Z}$ is the positive generator, which
is also the class of the restriction (to $X$) of the K\"{a}hler form
of ${\mathbb{C}}P^N$. In particular, in this case both $X$ and
$\Sigma$ are Fano manifolds.

\section{Hamiltonian fibrations} \cntrs
\label{S:ham-fib}
In what follows we will use the theory of symplectic and Hamiltonian
fibrations and their invariants. We refer the reader
to~\cite{GLS:symp-fibr, McD-Sa:Intro, McD-Sa:Jhol-2} for the
foundations.

Let $\pi: \widetilde{X} \to B$ be a smooth locally trivial fibration
with fiber $\Sigma$ and base $B$ which are both closed manifolds. We
will assume in addition that $B$ is a {\em simply connected} manifold.
Further, let $\widetilde{\Omega}$ be a closed $2$-form on
$\widetilde{X}$ such that the restriction $\Omega_b =
\widetilde{\Omega}|_{\Sigma_b}$ to each fiber $\Sigma_b =
\pi^{-1}(b)$, $b \in B$, is a symplectic form. Fix $b_0 \in B$, and
let $\omega_{_{\Sigma}}$ be a symplectic form on $\Sigma$ such that
$(\Sigma, \omega_{_{\Sigma}})$ is symplectomorphic to $(\Sigma_{b_0},
\Omega_{b_0})$. This structure is a special case of a so called
Hamiltonian fibration. It is well known that under these assumptions
all fibers $(\Sigma_b, \Omega_b)$ are symplectomorphic and in fact the
structure group of $\pi$ can be reduced to $\textnormal{Ham}(\Sigma,
\omega_{_{\Sigma}})$.

We will assume from now on that $B = S^2$. We identify $S^2 \cong
\mathbb{C}P^1$ in a standard way and view $S^2$ as a Riemann surface
whose complex structure is denoted by $j$.

\subsection{Holomorphic curves in Hamiltonian fibrations} \cntrsb
\label{sb:hol-curves-ham-fibr}
Let $\pi: (\widetilde{X}, \widetilde{\Omega}) \to S^2$ be a
Hamiltonian fibration as above. Denote by $T^{v}\widetilde{X} = \ker(D
\pi)$ the vertical part of the tangent bundle of $\widetilde{X}$. We
now introduce almost complex structures compatible with the fibration.
These are by definition almost complex structures $\widetilde{J}$ on
$\widetilde{X}$ with the following properties:
\begin{enumerate}
  \item The projection $\pi$ is $(\widetilde{J}, j)$--holomorphic.
  \item For every $z \in S^2$ the restriction $J_z$ of $\widetilde{J}$
   to $\Sigma_{z}$ is compatible with the symplectic form
   $\Omega_{z}$, i.e. $\Omega_z(J_z \xi, J_z \eta) = \Omega_z(\xi,
   \eta)$ for every $\xi, \eta \in T_z^v \widetilde{X}$, and
   $\Omega_z(\xi, J_z \xi)>0$ for every $0 \neq \xi \in T_z^v
   \widetilde{X}$.
\end{enumerate}
We denote the space of such almost complex structures by
$\widetilde{\mathcal{J}}(\pi, \widetilde{\Omega})$. 

Denote by $H_2^{\pi}\subset H_2(\widetilde{X};\mathbb{Z})$ be the set
of classes $\widetilde{A}$ such that $\pi_*(\widetilde{A}) = [S^2]$.
Given $\widetilde{A}$ and $\widetilde{J} \in
\widetilde{\mathcal{J}}(\pi, \widetilde{\Omega})$ denote by
$\mathcal{M}^{\frak s}(\widetilde{A}, \widetilde{J})$ the space of
$\widetilde{J}$--holomorphic sections in the class $\widetilde{A}$,
i.e. the space of maps $\widetilde{u}:S^2 \longrightarrow
\widetilde{X}$ with the following properties:
\begin{enumerate}
  \item $\widetilde{u}$ is $(j, \widetilde{J})$--holomorphic.
  \item $\widetilde{u}$ is a section, i.e. $\pi \circ \widetilde{u} =
   \textnormal{id}$.
  \item $\widetilde{u}_*[S^2] = \widetilde{A}$.
\end{enumerate}
Fix $z_0 \in S^2$ and fix an identification $(\Sigma,
\omega_{_{\Sigma}}) \approx (\Sigma_{z_0}, \Omega_{z_0})$.  The space
of sections comes with an evaluation map: $$ev_{\widetilde{J},z_0} :
\mathcal{M}^{\frak s}(\widetilde{A}, \widetilde{J}) \longrightarrow
\Sigma, \quad ev_{\widetilde{J},z_0}(\widetilde{u}) =
\widetilde{u}(z_0).$$

\subsubsection{Transversality} \label{sbsb:transversality} In order to
obtain regularity and transversality properties for the moduli spaces
of holomorphic sections and their evaluation maps we will need to work
with so called {\em regular} almost complex structures. Moreover,
since the moduli spaces of holomorphic sections are usually not
compact they do not carry fundamental classes and so the evaluation
maps do not induce in a straightforward way homology classes in their
target ($\Sigma$ in this case). The reason for non-compactness of
these moduli spaces is that a sequence of holomorphic sections might
develop bubbles in one of the fibers (see e.g.~\cite{McD-Sa:Jhol-2}).
The simplest way to overcome this difficulty is to make some
positivity assumptions on the fiber $\Sigma$ (called {\em
  monotonicity}). Under such conditions the moduli spaces of
holomorphic sections admits a nice compactification which makes it
possible to define homology classes induced by the evaluation maps.
Here is the relevant definition.
\begin{dfn} \label{df:monotone} Let $(\Sigma, \omega_{_{\Sigma}})$ be
   a symplectic manifold. Denote by $H_2^S(\Sigma) \subset H_2(\Sigma;
   \mathbb{Z})$ the image of the Hurewicz homomorphism $\pi_2(\Sigma)
   \longrightarrow H_2(\Sigma; \mathbb{Z})$.  Denote by $c_1^{\Sigma}
   \in H^2(\Sigma;\mathbb{Z})$ the first Chern class of the tangent
   bundle $(T\Sigma, J_{\Sigma})$, where $J_{\Sigma}$ is any almost
   complex structure compatible with $\omega_{_{\Sigma}}$. The
   symplectic manifold $(\Sigma, \omega_{_{\Sigma}})$ is called
   spherically monotone if there exists a constant $\lambda > 0$ such
   that for every $A \in H_2^S(\Sigma)$ we have $\omega_{_{\Sigma}}(A)
   = \lambda c_1^{\Sigma}(A)$. For example, if $\Sigma$ is a Fano
   manifold and $\omega_{_{\Sigma}}$ is a symplectic form with
   $[\omega_{_{\Sigma}}] = c_1^{\Sigma}$ then obviously $(\Sigma,
   \omega_{_{\Sigma}})$ is spherically monotone.
\end{dfn}

From now on we assume that the fiber $(\Sigma, \omega_{_{\Sigma}})$ of
$\pi:(\widetilde{X}, \widetilde{\Omega}) \longrightarrow S^2$ is
spherically monotone. Denote by $c_1^v = c_1(T^v\widetilde{X}) \in
H^2(\widetilde{X})$ the vertical Chern class, i.e.  the first Chern
class of the vertical tangent bundle of $\widetilde{X}$. The following
is proved in~\cite{Se:pi1, McD-Sa:Jhol-2}. There exists a {\em dense
  subset} $\widetilde{\mathcal{J}}_{\textnormal{reg}}(\pi,
\widetilde{\Omega}) \subset \widetilde{\mathcal{J}}(\pi,
\widetilde{\Omega})$ such that for every $\widetilde{J} \in
\widetilde{\mathcal{J}}_{\textnormal{reg}}(\pi, \widetilde{\Omega})$
and every $\widetilde{A} \in H_2^{\pi}$ the following holds:
\begin{enumerate}
  \item For every $\widetilde{A} \in H_2^{\pi}$, the moduli space
   $\mathcal{M}^{\frak{s}}(\widetilde{A}, \widetilde{J})$ of
   $\widetilde{J}$-holomorphic sections in the class $\widetilde{A}$
   is either empty or a smooth manifold of dimension
   \begin{equation*} dim_{\mathbb{R}} \mathcal{M}^{\frak
        s}(\widetilde{A}, \widetilde{J}) = \dim_{\mathbb{R}}\Sigma +
      2c_1^v(\widetilde{A}).
   \end{equation*}
   Moreover, $\mathcal{M}^{\frak{s}}(\widetilde{A}, \widetilde{J})$
   has a canonical orientation.
  \item The evaluation map $ev_{\widetilde{J},z_0} :
   \mathcal{M}^{\frak s}(\widetilde{A}, \widetilde{J}) \longrightarrow
   \Sigma$ is a pseudo-cycle (see~\cite{McD-Sa:Jhol-2} for the
   definition). In particular, its Poincar\'{e} dual gives a
   cohomology class $\mathcal{S}(\widetilde{A}; \widetilde{J}) \in
   H^d(\Sigma;\mathbb{Z})_{\scriptscriptstyle free} = H^d(\Sigma;
   \mathbb{Z})/\textnormal{torsion}$, where
   $d=-2c_1^v(\widetilde{A})$.  Moreover, the class
   $\mathcal{S}(\widetilde{A}; \widetilde{J})$ is independent of the
   regular $\widetilde{J}$ used to define it.  Therefore we will
   denote it from now on by $\mathcal{S}(\widetilde{A})$.
\end{enumerate}
We refer the reader to~\cite{McD-Sa:Jhol-2, Se:pi1} for more general
results on transversality.

The definition of regularity for $\widetilde{J} \in
\widetilde{\mathcal{J}}(\pi,\widetilde{\Omega})$ involves three
ingredients. The first is that the restriction $J_{z_0}$ of
$\widetilde{J}$ to $\Sigma = \Sigma_{z_0}$ is regular in the sense of
Chapter~3 of~\cite{McD-Sa:Jhol-2}, namely that the linearization of
the $\overline{\partial}_{J_{z_0}}$-operator at every
$J_{z_0}$-holomorphic curve in $\Sigma$ is surjective. (In addition
one has to require that certain evaluation maps for tuples of such
curves are mutually transverse.) The second ingredient is that (the
vertical part of) the $\overline{\partial}_{\widetilde{J}}$-operator
at every $\widetilde{J}$-holomorphic section is surjective. The third
one is that $ev_{\widetilde{J},z_0}$ is transverse to all
$J_{z_0}$-holomorphic bubble trees in $\Sigma$.

In practice, we will have to compute cohomology classes of the type
$\mathcal{S}(\widetilde{A}) = \mathcal{S}(\widetilde{A};
\widetilde{J})$ using a specific choice of $\widetilde{J}$ that
naturally appears in our context. It is not an easy task to decide
whether a given almost complex structure $\widetilde{J}$ is regular or
not. However, in some situations it is possible to compute some of the
classes $\mathcal{S}(\widetilde{A})$ by using almost complex
structures $\widetilde{J}$ that satisfy weaker conditions than
regularity. Criteria for verification of these conditions have been
developed in~\cite{Se:pi1} (see Proposition~7.11 there) and
in~\cite{McD-Sa:Jhol-2} (see Section~3.3 and~3.4 there). Below we will
actually not appeal to such criteria and use simpler arguments.

\subsection{The Seidel representation} \label{sb:seidel-rep} Let
$(\Sigma, \omega_{_{\Sigma}})$ be a closed monotone symplectic
manifold (see Definition~\ref{df:monotone}
is~\S\ref{sbsb:transversality}). Denote by $C_{\Sigma} \in \mathbb{N}$
the minimal Chern number, i.e. $$C_{\Sigma} = \min \{ c_1^{\Sigma}(A)
\mid A \in H_2^S, c_1^{\Sigma}(A)>0\}.$$ Denote by $\Lambda =
\mathbb{Z}[q^{-1}, q]$ the ring of Laurent polynomials. We endow
$\Lambda$ with a grading by setting $\deg(q) = 2C_{\Sigma}$. Let
$QH^*(\Sigma ; \Lambda) = (H^{\bullet}(\Sigma) \otimes \Lambda)^*$ be
the quantum cohomology of $\Sigma$, where the grading is induced from
both factors $H^{\bullet}(\Sigma)$ and $\Lambda$. We endow
$QH(\Sigma;\Lambda)$ with the quantum product $*$. The unity will be
denoted as usual by $1 \in QH^0(\Sigma;\Lambda)$. We refer the reader
to Chapter~11 of~\cite{McD-Sa:Jhol-2} for the definitions and
foundations of quantum cohomology. (Note however that our grading
conventions are slightly different than the ones
in~\cite{McD-Sa:Jhol-2}).) With our grading conventions we have:
$$QH^j(\Sigma ; \Lambda) = \bigoplus_{l
  \in \mathbb{Z}} H^{j-2 l C_{\Sigma}}(\Sigma)q^l.$$ We will need also
a coefficients extension of $QH(\Sigma; \Lambda)$. Denote
$\overline{\Lambda} = \mathbb{Z}[t^{-1}, t]$ the ring of Laurent
polynomials in the variable $t$, graded so that $\deg(t)=2$.  Consider
now $QH^*(\Sigma;\overline{\Lambda}) = (H^{\bullet}(\Sigma) \otimes
\overline{\Lambda})^*$, endowed with the quantum product $*$. We can
regard $\overline{\Lambda}$ as an algebra over $\Lambda$ using the
embedding of rings induced by $q \longmapsto t^{C_{\Sigma}}$. This
also induces an embedding of rings
$$QH^*(\Sigma; \Lambda) \hooklongrightarrow 
QH^*(\Sigma; \overline{\Lambda}).$$ We will therefore view from now on
$QH(\Sigma; \Lambda)$ as a subring of $QH(\Sigma; \overline{\Lambda})$
by setting $q = t^{C_{\Sigma}}$.

In~\cite{Se:pi1} Seidel associated to a Hamiltonian fibration $ \pi :
\widetilde{X} \longrightarrow S^2$ with fiber $\Sigma$ an invertible
element $\widetilde{S}(\pi) \in QH^0(\Sigma; \overline{\Lambda})$. We
refer the reader to~\cite{Se:pi1, McD-Sa:Jhol-2} for a detailed
account of this theory.  Here is a brief review of the main
construction.

Pick a regular almost complex structure $\widetilde{J} \in
\widetilde{\mathcal{J}}_{\textnormal{reg}}(\pi, \widetilde{\Omega})$.
Define a class:
\begin{equation} \label{eq:seidel-element-gnrl} \widetilde{S}(\pi) :=
   \sum_{\widetilde{A} \in H_2^{\pi}}
   \mathcal{S}(\widetilde{A};\widetilde{J}) \otimes
   t^{c_1^v(\widetilde{A})} \in QH^{0}(\Sigma;\overline{\Lambda}).
\end{equation}
Note that since the degree of
$\mathcal{S}(\widetilde{A},\widetilde{J})$ is
$-2c_1^v(\widetilde{A})$, a class $\widetilde{A} \in H^{\pi}_2$
contributes to the sum in~\eqref{eq:seidel-element-gnrl} only if
\begin{equation} \label{eq:cbound}
   2-2n \leq 2c_1^v(\widetilde{A}) \leq 0.
\end{equation}
The class $\widetilde{S}(\pi)$ is called the Seidel element of the
fibration $\pi: (\widetilde{X}, \widetilde{\Omega}) \longrightarrow
S^2$.

In what follows it will be more convenient to work with the more
``economical'' ring $\Lambda$ rather than $\overline{\Lambda}$.  We
will now define a normalized version of the Seidel element, denoted
$S(\pi)$, which lives in $QH(\Sigma; \Lambda)$. Fix a reference class
$\widetilde{A}_0 \in H_2^{\pi}$ and set $c_0(\pi) =
c_1^v(\widetilde{A}_0)$. Define now
\begin{equation} \label{eq:seidel-element-normalized-1} S(\pi) =
   t^{-c_0(\pi)} \widetilde{S}(\pi).
\end{equation}
Since any two classes in $H_2^{\pi}$ differ by a class in
$H_2^S(\Sigma)$, there exists a uniquely defined function $\nu:
H_2^{\pi} \to \mathbb{Z}$ such that
$$c_1^v(\widetilde{A}) = c_0(\pi) + \nu(\widetilde{A})C_{\Sigma}, 
\quad \forall \widetilde{A} \in H_2^{\pi}.$$ As $q = t^{C_{\Sigma}}$
we have:
\begin{equation} \label{eq:seidel-element} S(\pi) :=
   \sum_{\widetilde{A} \in H_2^{\pi}}
   \mathcal{S}(\widetilde{A};\widetilde{J}) \otimes
   q^{\nu(\widetilde{A})} \in QH^{-2c_0(\pi)}(\Sigma;\Lambda).
\end{equation}

By abuse of terminology we will call $S(\pi)$ also the Seidel element
of the fibration $\pi$. Of course the element $S(\pi)$ (as well as its
degree) depends on the choice of the reference section
$\widetilde{A}_0$, however different reference sections
$\widetilde{A}_0$ will result in elements that differ by a factor of
the type $q^r$ for some $r \in \mathbb{Z}$. In particular, many
algebraic properties of $S(\pi)$ (such as invertibility) do not depend
on this choice. We will therefore ignore this ambiguity from now on.

\subsubsection{Relations to Hamiltonian loops}
\label{sbsb:seidel-ham-loops}

An important feature of the theory is the connection between
Hamiltonian fibrations over $S^2$ with fiber $(\Sigma,
\omega_{_{\Sigma}})$ and $\pi_1(\textnormal{Ham}(\Sigma,
\omega_{_{\Sigma}}))$. To a loop based at the identity $\lambda =
\left \{ \varphi_t \right \}_{t \in S^1}$ in $Ham(\Sigma,
\omega_{\Sigma})$ one can associate a Hamiltonian fibration
$\pi_{\lambda}: \widetilde{M}_{\lambda} \rightarrow S^2$ as follows.
Let $D_{+}$ and $D_{-}$ be two copies of the unit disk in
$\mathbb{C}$, where the orientation on $D_{-}$ is reversed.  Define:
\begin{equation} \label{eq:fibr-from-loop} \widetilde{M}_{\lambda} =
   \Bigl( (\Sigma \times D_+ ) {\textstyle \coprod} (\Sigma \times
   D_-) \Bigr) / \sim, \quad \textnormal{where } (x , e^{2 \pi i t}_+)
   \sim (\varphi_t(x) , e^{2 \pi i t}_-).
\end{equation}
Identifying $S^2 \approx D_+ \cup_{\partial} D_-$ we obtain a
fibration $\pi: \widetilde{M}_{\lambda} \longrightarrow S^2$. As the
elements of $\lambda$ are symplectic diffeomorphisms, the form
$\omega_{_{\Sigma}}$ gives rise to a family of symplectic forms
$\{\Omega_z\}_{z \in S^2}$ on the fibers $\Sigma_z = \pi^{-1}(z)$ of
$\pi$. Moreover, $\pi:(\widetilde{M}_{\lambda}, \{\Omega_z\}_{z \in
  S^2}) \longrightarrow S^2$ is locally trivial. Since the elements of
$\lambda$ are in fact Hamiltonian diffeomorphisms it follows that the
family of fiberwise forms $\{\omega_z\}_{z \in S^2}$ can be extended
to a closed $2$-form $\widetilde{\Omega}$ on
$\widetilde{M}_{\lambda}$, i.e. $\widetilde{\Omega}|_{\Sigma_z} =
\Omega_z$ for every $z$. See~\cite{Se:pi1, McD-Sa:Jhol-2} for the
proofs. We therefore obtain from this construction a Hamiltonian
fibration $\pi: (\widetilde{M}_{\lambda}, \widetilde{\Omega})
\longrightarrow S^2$.

From the construction one can see that homotopic loops in
$\textnormal{Ham}(\Sigma, \omega_{_{\Sigma}})$ give rise to isomorphic
fibrations. We denote the isomorphism class of fibrations
corresponding to an element $\gamma \in \pi_1(Ham(\Sigma,
\omega_{_{\Sigma}}))$ by $\pi_{\gamma}$.

Conversely, if $\pi : (\widetilde{M}, \widetilde{\Omega})
\longrightarrow S^2$ is a Hamiltonian fibration with fiber $(\Sigma,
\omega_{_{\Sigma}})$ one can express $\widetilde{M}$ as a gluing of
two trivial bundles over the two hemispheres in $S^2$. The gluing map
would be a loop of Hamiltonian diffeomorphisms of $(\Sigma,
\omega_{_{\Sigma}})$.  Different trivializations lead to homotopic
loops. Thus the fibration $\pi$ determines a class $\gamma(\pi) \in
\pi_1(Ham(\Sigma, \omega_{_{\Sigma}}))$.

This correspondence has the following properties in relation to the
Seidel elements (see~\cite{Se:pi1} for the proofs):
$$\overline{S}(\pi_{\gamma_1 \cdot \gamma_2}) =
\overline{S}(\pi_{\gamma_1} ) \ast \overline{S}(\pi_{\gamma_2}), \quad
\forall \, \gamma_1, \gamma_2 \in \pi_1(Ham(\Sigma,
\omega_{_{\Sigma}})).$$ Here $*$ stands for the quantum product.  The
unit element $e \in \pi_1(\textnormal{Ham}(\Sigma,
\omega_{_{\Sigma}}))$ corresponds to the trivial fibration $\pi_{e} :
\Sigma \times S^2 \longrightarrow S^2$ and we have
$\overline{S}(\pi_e) = 1 \in QH(\Sigma;\Lambda)$. It follows that
$\overline{S}(\pi)$ is an invertible element in
$QH(\Sigma;\overline{\Lambda})$ for every $\pi$. The corresponding
homomorphism $$\overline{S} : \pi_1(Ham(\Sigma, \omega_{_{\Sigma}}))
\longrightarrow QH(\Sigma, \overline{\Lambda})^{\times}, \quad \gamma
\longmapsto \overline{S}(\pi_{\gamma})$$ (which by abuse of notation
we also denote by $\overline{S}$), where $QH(\Sigma,
\overline{\Lambda})^{\times}$ is the group of invertible elements in
$QH(\Sigma, \overline{\Lambda})$, is called the Seidel representation.

As mentioned before, for our purposes it would be more convenient to
work with the normalized version $S(\pi)$ of the Seidel element rather
than with $\overline{S}(\pi)$. We claim that any normalized Seidel
element $S(\pi)$ is invertible in $QH(\Sigma; \Lambda)$ (not just in
$QH(\Sigma; \overline{\Lambda})$). To see this, denote by $\gamma \in
\pi_1(\textnormal{Ham}(\Sigma))$ the homotopy class of loops
corresponding to the fibration $\pi$ (so that $\pi = \pi_{\gamma}$).
Denote by $\pi' = \pi_{\gamma^{-1}}$ the fibration corresponding to
the inverse of $\gamma$. Choose two reference sections
$\widetilde{A}_0$ and $\widetilde{A'}_0$ for $\pi$ and $\pi'$
respectively. The corresponding normalized Seidel elements are $S(\pi)
= t^{-c_0(\pi)} \widetilde{S}(\pi)$, $S(\pi') = t^{-c_0(\pi')}
\widetilde{S}(\pi')$.  Since $\widetilde{S}_{\pi} *
\widetilde{S}_{\pi'} = 1$ we have
$$S(\pi) * S(\pi') = t^{-c_0(\pi) -c_0(\pi')}.$$
But $S(\pi)$ and $S(\pi')$ both belong to the subring $QH(\Sigma;
\Lambda)$ of $QH(\Sigma; \overline{\Lambda})$, hence their product
$S(\pi) * S(\pi') \in QH(\Sigma; \Lambda)$ too. Thus
$t^{-c_0(\pi)-c_0(\pi')} = q^{r}$ for some $r \in \mathbb{Z}$. It
follows that $S(\pi)$ is invertible in $QH(\Sigma; \Lambda)$.

\section{From manifolds with small dual to Hamiltonian fibrations}
\cntrs \label{s:small-dual-fibr}
Let $X \subset \mathbb{C}P^N$ be a projective manifold with small
dual. Put $n=\dim_{\mathbb{C}}X$ and $k = \tdef(X) >0$. Since $X^*
\subset ({\mathbb{C}}P^N)^*$ has codimension $k+1 \geq 2$ we can find
a pencil of hyperplanes $\ell \subset (\mathbb{C}P^N)^{\ast}$ such
that $\ell$ does not intersect $X^*$. Consider the manifold
$$ \widetilde{X} = \left \{ (x,H) \mid H \in \ell, \; x \in H \right \} 
\subset X \times \ell.$$ Identify $\ell \cong \mathbb{C}P^1 \cong S^2$
in an obvious way. Denote by $$p : \widetilde{X} \longrightarrow X,
\quad \pi_{\ell} : \widetilde{X} \longrightarrow \ell \cong S^2$$ the
obvious projections. The map $p$ can be considered as the blowup of
$X$ along the base locus of the pencil $\ell$. The map $\pi_{\ell}$ is
a honest holomorphic fibration (without singularities) over $\ell
\cong \mathbb{C}P^1$ with fibers $\pi_{\ell}^{-1}(H) = X \cap H$.

Denote by $\omega_{_X}$ the symplectic form on $X$ induced from the
Fubini-Study K\"{a}hler form of ${\mathbb{C}}P^N$. Let $\omega_{S^2}$
be an area form on $S^2$ with $\int_{S^2} \omega_{S^2} = 1$.  Endow $X
\times S^2$ with $\omega_{_X} \oplus \omega_{S^2}$ and denote by
$\widetilde{\Omega}$ the restriction of $\omega_{_X} \oplus
\omega_{S^2}$ to $\widetilde{X} \subset X \times S^2$. The restriction
of $\widetilde{\Omega}$ to the fibers
$\widetilde{\Omega}|_{{\pi_{\ell}}^{-1}(H)}$, $H \in \ell$, coincides
with the symplectic forms $\omega_{_X}|_{X \cap H}$. Thus $\pi_{\ell}:
\widetilde{X} \longrightarrow S^2$ is a Hamiltonian fibration. Fix a
point $H_0 \in \ell$, and set $(\Sigma, \omega_{_{\Sigma}}) =
(\pi_{\ell}^{-1}(H_0), \omega_{_X}|_{X \cap H_0})$.

\begin{rem} \label{r:pencils} Different pencils $\ell \subset
   ({\mathbb{C}}P^N)^*$ with $\ell \cap X^* = \emptyset$ give rise to
   isomorphic Hamiltonian fibrations. This is so because the {\em
     real} codimension of $X^{\ast}$ is at least $4$ hence any two
   pencils $\ell, \ell'$ which do not intersect $X^{\ast}$ can be
   connected by a \emph{real} path of pencils in the complement of
   $X^*$. Thus the isomorphism class of the Hamiltonian fibration
   $\pi_{\ell}$, the element $\gamma(\pi_{\ell}) \in
   \pi_1(\textnormal{Ham}(\Sigma, \omega_{_{\Sigma}}))$, as well as
   the corresponding Seidel element $S(\pi_{\ell})$ can all be viewed
   as invariants of the projective embedding $X \subset
   {\mathbb{C}}P^N$.
\end{rem}

\begin{thm} \label{t:om-inv-gnrl} Let $X \subset \mathbb{C}P^N$ be an
   algebraic manifold with $\dim_{\mathbb{C}}(X) = n \geq 2$ and
   $\tdef(X)=k>0$.  Denote by $H_2^S(X) = \textnormal{image\,}
   \bigl(\pi_2(X) \longrightarrow H_2(X;\mathbb{Z})\bigr)\subset
   H_2(X;\mathbb{Z})$ the image of the Hurewicz homomorphism. Denote
   by $h \in H^2(X)$ the class dual to the hyperplane section. Assume
   that there exists $0< \lambda \in \mathbb{Q}$ such that $c_1^X(A) =
   \lambda h(A)$ for every $A \in H_2^S(X)$. Then the Seidel element
   of the fibration $\pi_{\ell} : \widetilde{X} \longrightarrow \ell$
   is:
   $$S(\pi_{\ell})=[\omega_{_{\Sigma}}] \in QH^2(\Sigma ; \Lambda).$$
   The degree of the variable $q \in \Lambda$ is
   $deg(q)=\frac{n+k}{2}$.
\end{thm}
The proof of this Theorem is given in~\S\ref{s:proof-om-inv}.
\begin{rem} \label{r:lambda} The condition $c_1^X(A) = \lambda h(A)$,
   $\forall A \in H_2^S$, implies that $\lambda = \tfrac{n+k+2}{2}$.
   Indeed, as explained in~\S\ref{S:PDS}, manifolds $X$ with small
   dual contain projective lines $S \subset X$ (embedded linearly in
   ${\mathbb{C}}P^N$) with $c_1^X(S) = \tfrac{n+k+2}{2}$. As $h(S) =
   1$ it follows that $\lambda = \tfrac{n+k+2}{2}$.
\end{rem}

\subsubsection*{Examples} 
Theorem~\ref{t:om-inv-gnrl} applies for example to algebraic manifolds
$X \subset {\mathbb{C}}P^N$ with small dual that satisfy one of the
following conditions:
\begin{enumerate}
  \item $b_2(X) = 1$. \label{i:b_2=1}
  \item More generally, the free part of $H_2^S(X)$ has rank $1$.
   \label{i:H_2^S=1}
\end{enumerate}
This is so because in both of these cases we must have $h = \lambda
c_1^X$ for some $\lambda \in \mathbb{Q}$. The fact that $\lambda > 0$
follows from the existence of rational curves $S \subset X$ with
$c_1^X(S) = \tfrac{n+k+2}{2}$ as explained in~\S\ref{S:PDS}.

Here is a concrete class of examples with $b_2(X)>1$ (hence different
than those in~\S\ref{S:Intro}) to which Theorem~\ref{t:om-inv-gnrl}
applies. Let $Y \subset {\mathbb{C}}P^m$ be any algebraic manifold
with $\pi_2(Y) = 0$ (or more generally with $h_Y(A)= 0$ for every $A
\in H_2^S(Y)$, where $h_Y$ is the Poincar\'{e} dual of the hyperplane
class on $Y$). Let $i: {\mathbb{C}}P^n \times {\mathbb{C}}P^m
\longrightarrow {\mathbb{C}}P^{(n+1)(m+1)-1}$ be the Segre embedding
and put $X = i({\mathbb{C}}P^n \times Y)$.  It is well known (see
Theorem~6.5 in~\cite{Te:dual-arxiv}) that
$$\tdef(X) \geq n-\dim_{\mathbb{C}}(Y),$$ hence if $n >
\dim_{\mathbb{C}}(Y)$, $X$ will have a small dual. Note that the
conditions of Theorem~\ref{t:om-inv-gnrl} are obviously satisfied.

One could generalize this example further by replacing
${\mathbb{C}}P^n$ with any manifold $Z$ satisfying $c_1^Z(A) = \lambda
h_Z(A)$ for every $A \in H_2^S(Z)$ for some $\lambda > 0$ and such
that $\tdef(Z) > \dim_{\mathbb{C}}(Y)$.  (See~\cite{Te:dual-arxiv} for
more on such examples.)

\begin{cor} \label{c:periodicity-gnrl} Under the assumptions of
   Theorem~\ref{t:om-inv-gnrl} we have: $$\widetilde{b}_j(X) =
   \widetilde{b}_{j+2}(X), \quad \widetilde{b}_j(\Sigma) =
   \widetilde{b}_{j+2}(\Sigma) \;\; \forall \, \, j \in \mathbb{Z},$$
   where the definition of $\widetilde{b}_j$ is given
   in~\eqref{eq:chmlg-cyc-grd} in~\S\ref{S:Intro}. Or, put in an
   unwrapped way, we have the following identities for $X$:
   \begin{align*}
      & b_j(X) + b_{j+n+k+2}(X) = b_{j+2}(X) + b_{j+n+k+4}(X), \;\;
      \forall \, \, 0 \leq j \leq n+k-1,  \\
      & b_{n+k}(X) = b_{n+k+2}(X) + 1, \quad b_{n+k+1}(X) =
      b_{n+k+3}(X)+ b_1(X),
   \end{align*}
   and the following ones for $\Sigma$:
   \begin{align*}
      & b_j(\Sigma) + b_{j+n+k}(\Sigma) = b_{j+2}(\Sigma) +
      b_{j+n+k+2}(\Sigma), \;\;
      \forall \, \, 0 \leq j \leq n+k-3,  \\
      & b_{n+k-2}(\Sigma) = b_{n+k}(\Sigma) + 1, \quad
      b_{n+k-1}(\Sigma) = b_{n+k+1}(\Sigma)+b_1(\Sigma).
   \end{align*}
\end{cor}
The proof is given in~\S\ref{s:prf-periodicity-gnrl}

\section{Proofs of theorem~\ref{t:om-inv-gnrl} and
  Theorems~\ref{mt:om-inv} and ~\ref{mt:ham-loop}}
\cntrs \label{s:proof-om-inv} 

As noted in the discussion after the statement of
Theorem~\ref{t:om-inv-gnrl},
Theorems~\ref{mt:om-inv},~\ref{mt:ham-loop} from~\S\ref{S:Intro} are
immediate consequences of Theorem~\ref{t:om-inv-gnrl}. Therefore we
will concentrate in this section in proving the latter. We will make
throughout this section the same assumptions as in
Theorem~\ref{t:om-inv-gnrl} and use here the construction and notation
of~\S\ref{s:small-dual-fibr}.

For a hyperplane $H \in ({\mathbb{C}}P^N)^*$ write $\Sigma_H = X \cap
H$. For a pencil $\ell \subset ({\mathbb{C}}P^N)^*$ denote by
$B_{\ell} = \Sigma_{H_0} \cap \Sigma_{H_1} \subset X$, ($H_0, H_1 \in
\ell$), its base locus. Recall that $p: \widetilde{X} \longrightarrow
X$ can be viewed as the blowup of $X$ along $B_{\ell}$. Denote by $E
\subset \widetilde{X}$ the exceptional divisor of this blowup.  The
restriction $p|_E : E \longrightarrow B_{\ell}$ is a holomorphic
fibration with fiber $\mathbb{C}P^1$.  Denote the homology class of
this fiber by $F \in H_2(\widetilde{X};\mathbb{Z})$.  Since
$\dim_{\mathbb{R}} B_{\ell} = 2n-4$, the map induced by inclusion
$H_2(X \setminus B_{\ell};\mathbb{Z}) \longrightarrow
H_2(X;\mathbb{Z})$ is an isomorphism, hence we obtain an obvious
injection $j: H_2(X;\mathbb{Z}) \longrightarrow
H_2(\widetilde{X};\mathbb{Z})$. The 2'nd homology of $\widetilde{X}$
is then given by
$$H_2(\widetilde{X};\mathbb{Z}) = j(H_2(X;\mathbb{Z})) \oplus
\mathbb{Z} F.$$ The $(2n-2)$'th homology of $\widetilde{X}$ fits into
the following exact sequence:
$$0 \longrightarrow \mathbb{Z}[E] 
\longrightarrow H_{2n-2}(\widetilde{X};\mathbb{Z})
\stackrel{p_*}{\longrightarrow} H_{2n-2}(X;\mathbb{Z}) \longrightarrow
0,$$ where the first map is induced by the inclusion. We obviously
have $p_* \circ j = \textnormal{id}$. Denote by $\widetilde{\Sigma}
\subset \widetilde{X}$ the proper transform of $\Sigma$ (with respect
to $p$) in $\widetilde{X}$. The intersection pairing between
$H_{2n-2}$ and $H_2$ in $\widetilde{X}$ is related to the one in $X$
as follows:
\begin{equation} \label{eq:inters-prod-wtldX}
   \begin{aligned}
      V \cdot j(A) = p_*(V) \cdot A, \;\;\;\; \forall \, V \in
      H_{2n-2}(\widetilde{X};\mathbb{Z}), \, A \in H_2(X;\mathbb{Z}),
      \\
      [\widetilde{\Sigma}] \cdot F = 1, \quad [E] \cdot F = -1, \quad
      [E]\cdot j(A) = 0, \;\;\;\;\forall A \in H_2(X;\mathbb{Z}).
   \end{aligned}
\end{equation}
Consider now the fibration $\pi_{\ell} : \widetilde{X} \longrightarrow
\ell$. The fiber over $H_0 \in \ell$ is precisely $\Sigma =
\Sigma_{H_0}$. It follows from~\eqref{eq:inters-prod-wtldX} that the
set of classes $H_2^{\pi_{\ell}}$ that represent sections of $\pi$
satisfies:
\begin{equation} \label{eq:H_2^pi}
   H_2^{\pi_{\ell}} \subset  \{ j(A) + dF \mid [\Sigma] \cdot A = 1-d\}.
\end{equation}

Denote by $J_0$ the standard complex structure of $X$ (coming from the
structure of $X$ as an algebraic manifold). Denote by $\mathcal{R}(X)
\subset H_2(X;\mathbb{Z})$ the positive cone generated by classes that
represent $J_0$-holomorphic rational curves in $X$, i.e.
$$\mathcal{R}(X) = \Bigl\{\sum a_i [C_i] \; \big| \; 
a_i \in \mathbb{Z}_{\geq 0}, \; C_i \subset X \; \textnormal{ is a
  rational } J_0\textnormal{-holomorphic curve} \Bigr\}.$$

\begin{lem} \label{l:curves-1} Let $\widetilde{A}= j(A)+dF \in
   H_2^{\pi_{\ell}}$, with $A \in H_2(X;\mathbb{Z})$, $d \in
   \mathbb{Z}$.  If $\mathcal{S}(\widetilde{A}) \neq 0$ then $A \in
   \mathcal{R}(X)$ and $d \leq 1$, with equality if and only if $A=0$.
\end{lem}
\begin{proof}
   Denote by $\widetilde{J_0}$ the standard complex structure on
   $\widetilde{X} \subset X \times \ell$, namely the complex structure
   induced from the standard complex structure $J_0 \oplus i$ on $X
   \times \ell$. Let $\widetilde{J_n}$ be a sequence of regular almost
   complex structures on $\widetilde{X}$ with $\widetilde{J_n}
   \longrightarrow \widetilde{J_0}$. Since $\mathcal{S}(\widetilde{A},
   \widetilde{J_n}) \neq 0$, there exist $\widetilde{J_n}$-holomorphic
   sections $u_n \in \mathcal{M}^{\frak s}(\widetilde{A},
   \widetilde{J_n})$. After passing to the limit $n \longrightarrow
   \infty$ we obtain by Gromov compactness theorem a (possibly
   reducible) $\widetilde{J_0}$-holomorphic curve $D \subset
   \widetilde{X}$ in the class $\widetilde{A}$. As $p : \widetilde{X}
   \longrightarrow X$ is $(\widetilde{J_0}, J_0)$-holomorphic it
   follows that $p(D)$ is a $J_0$-holomorphic rational curve, hence $A
   = p_*([D]) \in \mathcal{R}(X)$.

   Next, recall that $[\Sigma] \cdot A = 1-d$. But $\Sigma \subset X$
   is ample, hence $[\Sigma]\cdot A = 1-d \geq 0$ with equality if and
   only if $A=0$.
\end{proof}

The next lemma shows that when $d < 1$ the sections in the class
$\widetilde{A}$ do not contribute to the Seidel element
in~\eqref{eq:seidel-element}.
\begin{lem} \label{l:d<1} Let $\widetilde{A} = j(A) + dF \in
   H_2^{\pi_{\ell}}$ with $A \in H_2(X;\mathbb{Z})$ and $d<1$. Then
   $c_1^v(\widetilde{A}) > 0$. In particular, in view
   of~\eqref{eq:cbound}, $\widetilde{A}$ does not contribute to
   $S(\pi_{\ell})$.
\end{lem}
\begin{proof}
   Denote by $c_1^{\widetilde{X}}$ the first Chern class of (the
   tangent bundle of) $\widetilde{X}$ and by $c_1^X$ that of $X$.
   Since $\widetilde{X}$ is the blowup of $X$ along $B_{\ell}$, the
   relation between these Chern classes is given by:
   \begin{equation} \label{eq:c_1-wtldX}
      c_1^{\widetilde{X}} = p^*c_1^X - PD([E]),
   \end{equation}
   where $PD([E]) \in H^2(\widetilde{X})$ stands for the Poincar\'{e}
   dual of $[E]$. (See e.g.~\cite{Gr-Ha:alg-geom}.)

   Denote by $c_1^{\ell}$ the first Chern class of $\ell \cong
   \mathbb{C}P^1$.  Since $\widetilde{A}$ represents sections of
   $\pi_{\ell}$ we have: $$c_1^v(\widetilde{A}) =
   c_1^{\widetilde{X}}(\widetilde{A}) -
   \pi_{\ell}^*(c_1^{\ell})(\widetilde{A}) =
   c_1^{\widetilde{X}}(\widetilde{A}) - 2.$$ Together
   with~\eqref{eq:c_1-wtldX} and~\eqref{eq:inters-prod-wtldX} this
   implies:
   \begin{equation} \label{eq:c1-comp-1} c_1^v(\widetilde{A}) =
      p^*(c_1^X)(\widetilde{A}) - [E] \cdot \widetilde{A} - 2 =
      c_1^X(A) + d-2.
   \end{equation}
   By Lemma~\ref{l:curves-1}, $A \in \mathcal{R}(X)\subset H_2^S(X)$,
   hence by Remark~\ref{r:lambda} we have $$c_1^X(A) = \frac{n+k+2}{2}
   h(A) = \frac{n+k+2}{2} ([\Sigma] \cdot A) = \frac{n+k+2}{2}(1-d).$$
   Together with~\eqref{eq:c1-comp-1} we obtain:
   $$c_1^v(\widetilde{A}) = \frac{n+k}{2}(1-d)-1 \geq 
   \frac{n+k}{2}-1 > 0,$$ because $d<1$ and $n \geq 2$.
\end{proof}

We now turn to the case $\widetilde{A} = F$.  Let $b \in B_{\ell}$.
Define
$$\widetilde{u}_b : \ell \longrightarrow \widetilde{X}, \quad 
\widetilde{u}_b(z) = (b, z).$$ It is easy to see that
$\widetilde{u}_b$ is a $\widetilde{J_0}$-holomorphic section of
$\pi_{\ell}$ representing the class $F$. 
\begin{lem} \label{l:sect-class-F} The sections $\widetilde{u}_b$, $b
   \in B_{\ell}$, are the only $\widetilde{J_0}$-holomorphic sections
   in the class $F$, hence $\mathcal{M}^{\frak s}(F, \widetilde{J_0})
   = \{ \widetilde{u}_b \mid b \in B_{\ell}\}$. The evaluation map is
   given by
   $$ev_{\widetilde{J_0}, z_0}(\widetilde{u}_b) = b \in \Sigma$$ 
   and is an orientation preserving diffeomorphism between
   $\mathcal{M}^{\frak s}(F, \widetilde{J_0})$ and the base locus
   $B_{\ell}$.
\end{lem}
\begin{proof}
   Let $\widetilde{u} : \ell \longrightarrow \widetilde{X}$ be a
   $\widetilde{J_0}$-holomorphic section in the class $F$.  Write
   $\widetilde{u}(z) = (v(z), z) \in X \times \ell$. Due to our choice
   of $\widetilde{J_0}$, $v$ is a $J_0$-holomorphic map. Since $p_*(F)
   = 0$ the map $v = p \circ u: \ell \longrightarrow X$ must be
   constant, say $v(z) \equiv b$, $b \in X$. But $v(z) \in \Sigma_{z}$
   for every $z \in \ell$. It follows that $b \in \cap_{z \in \ell}
   \Sigma_z = B_{\ell}$. The rest of the statements in the lemma are
   immediate.
\end{proof}

We are now ready for the
\begin{proof}[Proof of Theorem~\ref{t:om-inv-gnrl}]
   In view of~\eqref{eq:H_2^pi} and
   Lemmas~\ref{l:curves-1},~\ref{l:d<1}, the only class that
   contributes to the Seidel element $S(\pi_{\ell})$ is $F$, hence: 
   $$S(\pi_{\ell}) = \mathcal{S}(F) \in QH^2(\Sigma;\Lambda).$$
   (We take $F$ to be the reference class of sections and note that
   $c_1^v(F)=-1$.)

   In order to evaluate $\mathcal{S}(F)$ we need to compute
   $\mathcal{S}(F, \widetilde{J})$ for a regular $\widetilde{J}$.  We
   first claim that there exists a neighborhood $\mathcal{U}$ of
   $\widetilde{J_0}$ inside $\widetilde{\mathcal{J}}(\pi_{\ell},
   \widetilde{\Omega})$ such that for every $\widetilde{J} \in
   \mathcal{U}$ the space $\mathcal{M}^{\frak{s}}(F, \widetilde{J})$
   is compact.

   To see this, first note that $\widetilde{\Omega}$ is a genuine
   symplectic form on $\widetilde{X}$ and that $\widetilde{J_0}$ is
   tamed by $\widetilde{\Omega}$ (i.e. $\Omega(v, \widetilde{J_0}v)>0$
   for all non-zero vectors $v \in T \widetilde{X}$ be they vertical
   or not). Hence there is a neighborhood $\mathcal{U}$ of
   $\widetilde{J_0}$ in $\widetilde{\mathcal{J}}(\pi_{\ell},
   \widetilde{\Omega})$ such that every $\widetilde{J} \in
   \mathcal{U}$ is tamed by $\widetilde{\Omega}$. Next note that
   $\widetilde{\Omega}$ defines an integral (modulo torsion)
   cohomology class $[\widetilde{\Omega}] \in
   H^2(\widetilde{X})_{\scriptscriptstyle free}$ and that
   $\widetilde{\Omega}(F) = 1$ (see~\S\ref{s:small-dual-fibr}).  It
   follows that $F$ is a class of minimal positive area for
   $\widetilde{\Sigma}$. Therefore, for $\widetilde{J}$ tamed by
   $\widetilde{\Omega}$, a sequence of $\widetilde{J}$-holomorphic
   rational curves in the class $F$ cannot develop bubbles. By Gromov
   compactness $\mathcal{M}^{\frak{s}}(F, \widetilde{J})$ is compact.

   Next we claim that $\widetilde{J_0}$ is a regular almost complex
   structure in the sense of the general theory of pseudo-holomorphic
   curves (see Chapter~3 in~\cite{McD-Sa:Jhol-2}). To see this recall
   the following regularity criterion (see Lemma~3.3.1
   in~\cite{McD-Sa:Jhol-2}): {\sl let $(M, \omega)$ be a symplectic
     manifold and $J$ an {\em integrable} almost complex structure.
     Then $J$ is regular for a $J$-holomorphic curve $u:\mathbb{C}P^1
     \longrightarrow M$ if every summand of the holomorphic bundle
     $u^* TM \to \mathbb{C}P^1$ (in its splitting to a direct sum of
     line bundles) has Chern number $\geq -1$.}  Applying this to our
   case, a simple computation shows that for every $\widetilde{u}_b
   \in \mathcal{M}^{\frak{s}}(F, \widetilde{J_0})$ we have
   $$\widetilde{u}_b^* T \widetilde{X} = \mathcal{O}_{\ell}(2) \oplus
   \mathcal{O}_{\ell}^{\oplus (n-2)} \oplus \mathcal{O}_{\ell}(-1),$$
   hence $\widetilde{J_0}$ is regular for all $\widetilde{u} \in
   \mathcal{M}^{\frak{s}}(F, \widetilde{J_0})$.

   Pick a regular almost complex structure $\widetilde{J} \in
   \widetilde{\mathcal{J}}_{\textnormal{reg}}(\pi, \widetilde{\Omega})
   \cap \mathcal{U}$ which is close enough to $\widetilde{J_0}$. By
   the standard theory of pseudo-holomorphic
   curves~\cite{McD-Sa:Jhol-2} the evaluation maps
   $ev_{\widetilde{J},z_0}$ and $ev_{\widetilde{J_0},z_0}$ are
   cobordant, hence give rise to cobordant pseudo-cycles. Moreover by
   what we have seen before this cobordism can be assumed to be
   compact (and the pseudo-cycles are in fact cycles). It follows that
   the homology class $(ev_{\widetilde{J}, z_0})_*
   [\mathcal{M}^{\frak{s}}(F,\widetilde{J})]$ equals to
   $(ev_{\widetilde{J_0}, z_0})_*
   [\mathcal{M}^{\frak{s}}(F,\widetilde{J_0})] = [B_{\ell}]$. Putting
   everything together we obtain:
   $$S(\pi_{\ell}) = \mathcal{S}(F, \widetilde{J}) = PD([B_{\ell}]) = 
   [\omega_{_{\Sigma}}].$$
\end{proof}

\section{Subcriticality and projective defect} \cntrs
\label{S:subcri-def}

Here we discuss symplectic and topological aspects of manifolds $X$
with small dual that have to do with the structure of the (affine)
Stein manifold obtained after removing from $X$ a hyperplane section.
Some of the results of this section should be known to experts, but we
could not find them in explicit form in the literature. We therefore
state the results and include their proofs.

Let $ Y \subset \mathbb{C}^N$ be a Stein manifold. The study of Morse
theory on Stein manifolds was initiated in the classical
paper~\cite{An-Fr:Lefschetz-1} of Andreotti and Frankel and in its
sequel~\cite{An-Fr:Lefschetz-2}. Further aspects of Morse theory as
well as symplectic topology on Stein manifolds were studied by various
authors~\cite{El-Gr:convex, El:Stein, El:psh, Bi-Ci:Stein}. In this
context it is important to remark that by a result of
Eliashberg-Gromov~\cite{El-Gr:convex, El:psh}, Stein manifolds $Y$
admit a canonical symplectic structure $\widehat{\omega}_Y$ (see
also~\cite{Bi-Ci:Stein}).

A function $\varphi: Y \longrightarrow \mathbb{R}$ is called
plurisubharmonic (p.s.h in short) if the form $\Omega = -d
d^{\mathbb{C}} \varphi$ is a K\"{a}hler form on Y. Here
$d^{\mathbb{C}} \varphi= d \varphi \circ J$, where $J$ is the complex
structure of $Y$. A plurisubharmonic function $\varphi: Y
\longrightarrow \mathbb{R}$ is called exhausting if it is proper and
bounded from below. For a plurisubharmonic Morse function $\varphi: Y
\longrightarrow \mathbb{R}$ denote
$$\textnormal{ind}_{\max}(\varphi)= \max \left
   \{\textnormal{ind}_z(\varphi) \mid z \in \textnormal{Crit}(\varphi)
\right \},$$ where $\textnormal{ind}_z(\varphi)$ is the Morse index of
the critical point $z \in \textnormal{Crit}(\varphi)$. A fundamental
property of plurisubharmonic Morse functions $\varphi$ is that
$\textnormal{ind}_{\max}(\varphi) \leq \dim_{\mathbb{C}}(Y)$.
(Various proofs of this can be found in e.g.~\cite{An-Fr:Lefschetz-1,
  El-Gr:convex, El:Stein, El:psh}.)

A Stein manifold $Y$ is called {\em subcritical} if it admits an
exhausting plurisubharmonic Morse function $\varphi: Y \longrightarrow
\mathbb{R}$ with $\textnormal{ind}_{\max}(\varphi) <
\dim_{\mathbb{C}}(Y)$ for every $z \in \textnormal{Crit}(\varphi)$.
Otherwise we call $Y$ {\em critical}. The subcriticality index
$\textnormal{ind}(Y)$ of a Stein manifold $Y$ is defined by
$$\textnormal{ind}(Y) := \min \left \{ \textnormal{ind}_{\max}(\varphi) 
   \mid \varphi: Y \longrightarrow \mathbb{R} \textrm{ is a p.s.h
     exhausting Morse function} \right \}.$$ Thus we have $0 \leq
\textnormal{ind}(Y) \leq \dim_{\mathbb{C}}(Y)$, and $Y$ is subcritical
iff $\textnormal{ind}(Y) < \dim_{\mathbb{C}}(Y)$.

Subcritical Stein manifolds $Y$ have special symplectic properties
(when endowed with their canonical symplectic structure
$\widehat{\omega}_Y$). For example, every compact subset $A \subset Y$
is Hamiltonianly displaceable. In particular, whenever well defined,
the Floer homology $HF(L_1, L_2)$ of every pair of compact Lagrangian
submanifolds $L_1, L_2 \subset Y$ vanishes.  This in turn implies
strong topological restrictions on the Lagrangian submanifolds of $Y$
(see e.g.~\cite{Bi-Ci:Stein}). Such considerations will play an
important role in~\S\ref{s:further} below.

The main result in this context is the following.

\begin{thm} \label{mt:subcrit} Let $X \subset \mathbb{C}P^N$ be a
   projective manifold with small dual and let $\Sigma \subset X $ be
   a smooth hyperplane section of $X$.  Then the Stein manifold $X
   \setminus \Sigma$ is subcritical. In fact:
   $$\textnormal{ind}(X \setminus \Sigma) \leq dim_{\mathbb{C}}(X) -
   \tdef(X).$$
\end{thm}
See Theorem~\ref{t:subcrit-small-dual} for a partial converse to this
theorem.

Theorem~\ref{mt:subcrit} can be easily proved using the the theory
developed in~\cite{An-Fr:Lefschetz-2}. Below we give an alternative
proof.
\begin{proof}
   Write $n = \dim_{\mathbb{C}}(X)$, $k = \tdef(X)$ and assume $k>0$.
   Denote by $Y = X \setminus \Sigma \subset {\mathbb{C}}P^N \setminus
   H = \mathbb{C}^N$. Denote by $h(\cdot, \cdot)$ the standard
   Hermitian form of $\mathbb{C}^N$, by $(\cdot, \cdot) =
   \textnormal{Re} h(\cdot, \cdot)$ the standard scalar product and by
   $| \cdot |$ the standard Euclidean norm on $\mathbb{C}^N$. Fix a
   point $w_0 \in \mathbb{C}^N$ and define $\varphi_w : Y
   \longrightarrow \mathbb{R}$ to be the function
   $$\varphi_{w_0}(z):= | z-w_0 |^2, \quad \forall z \in Y.$$ 
   By standard arguments, for a generic point $w_0 \in \mathbb{C}^N$,
   $\varphi_{w_0}$ is an exhausting plurisubharmonic Morse function.
   It is well known (see e.g.~\cite{Vo:hodge-book-2}) that $z_0 \in Y$
   is a critical point of $\varphi_{w_0}$ if and only if
   $\overrightarrow{w_0z_0} \perp T_{z_0}Y$. In order to compute the
   Hessian of $\varphi_{w_0}$ at a critical point $z_0$ we need the
   second fundamental form. We will follow here the conventions
   from~\cite{Vo:hodge-book-2}. Denote by $\gamma: Y \longrightarrow
   \textnormal{Gr}(n,N)$ the Gauss map, $\gamma(x) = T_x Y$.  Consider
   the differential of this map $D\gamma_x: T_x Y \longrightarrow
   T_{T_x Y}\mathbb{C}^N = \hom (T_x Y, \mathbb{C}^N/T_x Y)$. This map
   induces a symmetric bilinear form:
   $$\Phi: S^2 T_x Y \longrightarrow \mathbb{C}^N / T_x Y$$ 
   which is called the second fundamental form.

   As $h(v,\overrightarrow{w_0 z_0}) = 0$ for every $v \in T_{z_0} Y$
   we can define a symmetric complex bilinear form:
   $$G:S^2 T_{z_0} Y \longrightarrow \mathbb{C}, \quad G(u,v) = 
   h(\Phi(u,v), \overrightarrow{w_0 z_0}).$$ A standard computation
   (see e.g.~\cite{Vo:hodge-book-2}) shows that the Hessian of
   $\varphi_{w_0}$ is given by:
   \begin{equation} \label{eq:hess-phi}
      \textnormal{Hess}_{z_0}\varphi_{w_0}(u,v) = 2( (u,v) +
      \textnormal{Re} G(u,v)), \; \forall u,v \in T_{z_0}Y.
   \end{equation}

   Next, by a result of Katz we have: $\textnormal{rank}_{\mathbb{C}}
   G \leq n - k$. (See~\cite{Te:dual, Te:dual-arxiv} and the
   references therein, e.g. expos\'{e} XVII by N.~Katz
   in~\cite{SGA-7-II}. See also~\cite{Gr-Ha:alg-geom-loc-diff}.)  It
   follows that $\dim_{\mathbb{\mathbb{R}}} \ker (\textnormal{Re} G)
   \geq 2k$.

   Denote the non-zero eigenvalues of $\textnormal{Re} G$ (in some
   orthonormal basis) by $\lambda_i$, $i=1, \ldots, r$, with $r \leq
   2n-2k$. It is well known that for real symmetric bilinear forms
   that appear as the real part of complex ones (e.g. $\textnormal{Re}
   G$) the following holds: $\lambda$ is an eigenvalue if and only if
   $-\lambda$ is an eigenvalue (see~\cite{Vo:hodge-book-2}), moreover
   the multiplicities of $\lambda$ and $-\lambda$ are the same. (See
   e.g.~\cite{Vo:hodge-book-2} for a proof.) It follows that the
   number of negative $\lambda_i$'s can be at most $n-k$.

   Coming back to~\eqref{eq:hess-phi}, the eigenvalues of
   $\textnormal{Hess}_{z_0} \varphi_{w_0}$ are of the form $1 +
   \lambda$ with $\lambda$ being an eigenvalue of $\textnormal{Re} G$.
   It follows that the number of negative eigenvalues of
   $\textnormal{Hess}_{z_0} \varphi_{w_0}$ is at most $n-k$.  This
   shows that $\textnormal{ind}_{z_0}(\varphi_{w_0}) \leq n-k$ for
   every $z_0 \in \textnormal{Crit}(\varphi_{w_0})$. In particular,
   $\textnormal{ind}(Y) \leq n-k$.
\end{proof}

Using standard arguments one gets from Theorem~\ref{mt:subcrit} the
following version of the Lefschetz hyperplane theorem for manifolds
with small dual, which was previously known and proved by other
methods in~\cite{La-St:top}:

\begin{cor} \label{c:Lef-small-dual} Let $X \subset {\mathbb{C}}P^N$
   be an algebraic manifold with $dim_{\mathbb{C}}X =n $ and
   $\tdef(X)=k$ and let $\Sigma \subset X$ be a smooth hyperplane
   section. Denote by $i : \Sigma \longrightarrow X$ the inclusion.
   The induced maps $i_* : H_j(\Sigma; \mathbb{Z}) \longrightarrow
   H_j(X; \mathbb{Z})$ and $i_*: \pi_j(\Sigma,*) \longrightarrow
   \pi_j(X,*)$ are:
   \begin{enumerate}
     \item Isomorphisms for $j < n+k-1$.
     \item Surjective for $j=n+k-1$.
   \end{enumerate}
   Similarly, the restriction map $i^*: H^j(X;\mathbb{Z})
   \longrightarrow H^j(\Sigma;\mathbb{Z})$ is an isomorphism for every
   $j < n+k-1$ and injective for $j=n+k-1$.
\end{cor}

Another consequence is the following refinement of the hard Lefschetz
theorem.
\begin{cor} \label{c:top-period} Let $X \subset {\mathbb{C}}P^N$ be as
   in Corollary~\ref{c:Lef-small-dual}. Denote by $\omega$ the
   K\"{a}hler form on $X$ induced from the standard K\"{a}hler form of
   ${\mathbb{C}}P^N$. Then the map
   $$L : H^j(X;\mathbb{R}) \longrightarrow H^{j+2}(X;\mathbb{R}),
   \quad L(a) = a \cup [\omega],$$ is an isomorphism for every $n-k-1
   \leq j \leq n+k-1$.
\end{cor}
\begin{proof}
   This follows from Corollary~\ref{c:Lef-small-dual} together with
   the Hard Lefschetz theorem applied both to $\Sigma$ and $X$.
\end{proof}

\section{Proof of Corollary~\ref{c:periodicity-gnrl}} \cntrs
\label{s:prf-periodicity-gnrl}

The quantum cohomology of $\Sigma$ can be written additively (as a
vector space) as
$$QH^j(\Sigma; \Lambda) \cong \bigoplus_{l \in \mathbb{Z}} 
H^{j+2C_{\Sigma} l}(\Sigma).$$ By Theorem~\ref{t:om-inv-gnrl},
$[\omega_{_{\Sigma}}] \in QH^2(\Sigma ; \Lambda)$ is invertible with
respect to the quantum product $*$, hence the map $$(-)*
[\omega_{_{\Sigma}}] : QH^j(\Sigma;\Lambda) \longrightarrow
QH^{j+2}(\Sigma;\Lambda), \quad a \longmapsto a*[\omega_{_{\Sigma}}]$$
is an isomorphism for every $j \in \mathbb{Z}$. The statement about
$\widetilde{b}_j(\Sigma)$ follows immediately.

We now turn to the proof of the statement about $\widetilde{b}_j(X)$.
First recall that $2C_{\Sigma} = n+k$ and $2C_{X} = n+k+2$. We will
show now that for every $0 \leq j \leq n+k+1$ we have
$\widetilde{b}_j(X) = \widetilde{b}_{j+2}(X)$.

\subsubsection*{Step 1. Assume $j \leq n+k-4$.} By
Corollary~\ref{c:Lef-small-dual}, $b_j(\Sigma) = b_j(X)$ and
$b_{j+2}(\Sigma) = b_{j+2}(X)$.  We claim that
\begin{equation} \label{eq:betti-I} b_{j+n+k}(\Sigma) =
   b_{j+n+k+2}(X), \quad b_{j+n+k+2}(\Sigma) = b_{j+n+k+4}(X).
\end{equation}
Indeed, by Corollary~\ref{c:Lef-small-dual}, $b_{n-j-k-2}(\Sigma) =
b_{n-j-k-2}(X)$, hence the first equation in~\eqref{eq:betti-I}
follows from Poincar\'{e} duality for $\Sigma$ and $X$. The proof
of the second equality is similar.

It follows that 
\begin{align*}
   \widetilde{b}_j(X) & =  b_j(X) + b_{j+n+k+2}(X) =
   b_j(\Sigma) + b_{j+n+k}(\Sigma) = \widetilde{b}_j(\Sigma) \\ 
   & = \widetilde{b}_{j+2}(\Sigma) = b_{j+2}(\Sigma) +
   b_{j+n+k+2}(\Sigma) = b_{j+2}(X) + b_{j+n+k+4}(X) =
   \widetilde{b}_{j+2}(X).
\end{align*}
   
\subsubsection*{Step 2. Assume $n+k-3 \leq j \leq n+k-1$.}
In this case we have $\widetilde{b}_j(X) = b_{j}(X)$ and
$\widetilde{b}_{j+2}(X) = b_{j+2}(X)$ and the equality between the
two follows from Corollary~\ref{c:top-period}.

\subsubsection*{Step 3. Assume $j = n+k$.}
We have to prove that $b_{n+k}(X) = b_0(X) + b_{n+k+2}(X)$.  By
Poincar\'{e} duality this is equivalent to showing that $b_{n-k}(X) =
b_0(X) + b_{n-k-2}(X)$. The last equality is, by
Corollary~\ref{c:Lef-small-dual}, equivalent to $b_{n-k}(\Sigma) =
b_0(\Sigma) + b_{n-k-2}(\Sigma)$. Applying Poincar\'{e} duality on
$\Sigma$ the latter becomes equivalent to $b_{n+k-2}(\Sigma) =
b_0(\Sigma) + b_{n+k}(\Sigma)$. But this has already been proved since
$b_{n+k-2}(\Sigma) = \widetilde{b}_{n+k-2}(\Sigma) =
\widetilde{b}_{n+k}(\Sigma) = b_0(\Sigma) + b_{n+k}(\Sigma)$.

\subsubsection*{Step 3. Assume $j = n+k+1$.} The proof in this case
is very similar to the case $j=n+k$. We omit the details.
\Qed

\section{Further results} \label{s:further}

As we have seen above the algebraic geometry of manifolds with small
dual is intimately connected with their symplectic topology. Here we
add another ingredient which has to do with Lagrangian submanifolds.
Below we will use the following notation. For an algebraic manifold $X
\subset {\mathbb{C}}P^N$ and an algebraic submanifold $\Sigma \subset
X$ we denote by $\omega_{_X}$ and $\omega_{_{\Sigma}}$ the
restrictions of the standard K\"{a}hler form of ${\mathbb{C}}P^N$ to
$X$ and to $\Sigma$ respectively.

The following theorem follows easily by combining results
from~\cite{An-Fr:Lefschetz-2} with the fact that vanishing cycles can
be represented by Lagrangian spheres~\cite{Ar:monodromy, Do:polynom,
  Se:thesis}. (See also~Theorem~K in~\cite{Bi:ICM2002},
and~Theorem~2.1 in ~\cite{Bi:ECM2004}.)
\begin{thm} \label{t:def-0-lag-sphere} Let $X\subset {\mathbb{C}}P^N$
   be an algebraic manifold and $\Sigma \subset X$ a hyperplane
   section. If $\tdef(X)=0$ then $(\Sigma, \omega_{_{\Sigma}})$
   contains a (embedded) Lagrangian sphere.
\end{thm}
Thus we can detect manifolds with small dual (i.e. $\tdef>0$) by
methods of symplectic topology e.g. by showing that their hyperplane
sections do not contain Lagrangian spheres.

In some situations we also have the converse to
Theorem~\ref{t:def-0-lag-sphere}.
\begin{thm} \label{t:no-lag-spheres} Let $\Sigma \subset X \subset
   {\mathbb{C}}P^N$ be as in Theorem~\ref{t:om-inv-gnrl} and assume in
   addition that $\dim_{\mathbb{C}}(\Sigma) \geq 3$. Then the
   symplectic manifold $(\Sigma, \omega_{_{\Sigma}})$ contains no
   Lagrangian spheres.
\end{thm}

\begin{rem} \label{r:vanishing-cycles} Note that from the results
   of~\cite{An-Fr:Lefschetz-2} it follows that the (homological)
   subgroup of vanishing cycles $V_{n-1} \subset H_{n-1}(\Sigma)$ of
   $\Sigma$ is trivial (here, $n-1 = \dim_{\mathbb{C}} \Sigma$).
   Theorem~\ref{t:no-lag-spheres}, asserting that $\Sigma$ has no
   Lagrangian spheres, is however stronger. Indeed, it is not known
   whether or not every Lagrangian sphere comes from a vanishing
   cycle.  Moreover in some cases Lagrangian spheres do exists but are
   null-homologous. (Put differently, in general it is not possible to
   use purely topological methods to prove non-existence of Lagrangian
   spheres.)
\end{rem}

\begin{proof}[Proof of Theorem~\ref{t:no-lag-spheres}]
   Suppose by contradiction that $L \subset (\Sigma,
   \omega_{_{\Sigma}})$ is a Lagrangian sphere. We will use now the
   theory of Lagrangian Floer cohomology for in order to arrive at a
   contradiction. More specifically, we will use here a particular
   case of the general theory that works for so called {\em monotone
     Lagrangian submanifolds}. We will take $\mathbb{Z}_2$ as the
   ground ring and work with the self Floer cohomology of $L$, denoted
   $HF(L,L)$, with coefficients in the Novikov ring
   $\Lambda_{\mathbb{Z}_2} = \mathbb{Z}_2[q,q^{-1}]$. This ring is
   graded so that the variable $q$ has degree $\deg(q) = N_L$, where
   $N_L$ is the minimal Maslov number of $L$. We refer the reader
   to~\cite{Oh:HF1, Oh:spectral, Bi-Co:rigidity, Bi-Co:Yasha-fest} for
   the foundations of this theory.

   Since $L$ is simply connected, the assumptions on $\Sigma$ and $X$
   imply that $L \subset \Sigma$ is a monotone Lagrangian submanifold
   and its minimal Maslov number is $N_L = 2C_{\Sigma} = n+k$. (Here,
   as in Theorem~\ref{t:om-inv-gnrl}, $k = \tdef(X) \geq 1$.)  Under
   these circumstances it is well known that the self Floer homology
   of $L$, $HF(L,L)$ is well defined and moreover we have an
   isomorphism of graded $\Lambda_{\mathbb{Z}_2}$-modules:
   $$HF^*(L,L) \cong (H^{\bullet}(L;\mathbb{Z}_2) \otimes 
   \Lambda_{\mathbb{Z}_2})^*.$$ Since $L$ is a sphere of dimension
   $\dim_{\mathbb{R}}(L) \geq 3$ this implies that
   \begin{equation} \label{eq:HF-1}
      HF^0(L,L) \cong \mathbb{Z}_2, \quad 
      HF^2(L,L) \cong H^2(L;\mathbb{Z}_2) = 0.
   \end{equation}
   
   Denote by $QH(\Sigma;\Lambda_{\mathbb{Z}_2})$ the modulo-$2$
   reduction of $QH(\Sigma; \Lambda)$ (obtained by reducing the ground
   ring $\mathbb{Z}$ to $\mathbb{Z}_2$). By
   Theorem~\ref{t:om-inv-gnrl}, $[\omega_{_{\Sigma}}] \in
   QH^2(\Sigma;\Lambda)$ is an invertible element, hence its
   modulo-$2$ reduction, say $\alpha \in
   QH^2(\Sigma;\Lambda_{\mathbb{Z}_2})$ is invertible too.

   We now appeal to the quantum module structure of $HF(L,L)$
   introduced in~\cite{Bi-Co:rigidity, Bi-Co:Yasha-fest,
     Bi-Co:qrel-long}. By this construction, $HF(L,L)$ has a structure
   of a graded module over the ring $QH(\Sigma;
   \Lambda_{\mathbb{Z}_2})$ where the latter is endowed with the
   quantum product. We denote the module action of
   $QH^*(\Sigma;\Lambda_{\mathbb{Z}_2})$ on $HF^*(L,L)$ by
   $$QH^i(\Sigma;\Lambda_{\mathbb{Z}_2})
   \otimes_{\Lambda_{\mathbb{Z}_2}} HF^j(L,L) \longrightarrow
   HF^{i+j}(L,L), \quad a \otimes x \longmapsto a \circledast x, \;\;
   i, j \in \mathbb{Z}.$$

   Since $\alpha \in QH^2(\Sigma;\Lambda_{\mathbb{Z}_2})$, $\alpha$
   induces an isomorphism $\alpha \circledast (-) : HF^*(L,L)
   \longrightarrow HF^{*+2}(L,L)$. This however, is impossible (e.g
   for $*=0$) in view of~\eqref{eq:HF-1}. Contradiction.
\end{proof}

\begin{cor} \label{c:sig-hyp-sect} Let $\Sigma$ be an algebraic
   manifold with $\dim_{\mathbb{C}}(\Sigma) \geq 3$ and $b_2(\Sigma) =
   1$. Suppose that $\Sigma$ can be realized as a hyperplane section
   of a projective manifold $X \subset {\mathbb{C}}P^N$ with small
   dual.  Then in any other realization of $\Sigma$ as a hyperplane
   section of a projective manifold $X' \subset {\mathbb{C}}P^{N'}$ we
   have $\tdef(X') > 0$.  In fact, $\tdef(X') = \tdef(X)$.
\end{cor}

\begin{proof}
   Let $\omega_{_{\Sigma}}$ be the restriction to $\Sigma$ (via
   $\Sigma \subset X \subset {\mathbb{C}}P^N$) of the standard
   symplectic structure of ${\mathbb{C}}P^N$. Similarly let
   $\omega_{_{\Sigma}}'$ the restriction to $\Sigma$ (via $\Sigma
   \subset X' \subset {\mathbb{C}}P^{N'}$) of the standard symplectic
   structure of ${\mathbb{C}}P^{N'}$. 

   Since $b_2(\Sigma)=1$, it follows from Lefschetz theorem that
   $b_2(X)=1$. Thus $X$ satisfies the conditions of
   Theorem~\ref{t:om-inv-gnrl} (see the discussion after
   Theorem~\ref{t:om-inv-gnrl}). By Theorem~\ref{t:no-lag-spheres} the
   symplectic manifold $(\Sigma, \omega_{_{\Sigma}})$ does not contain
   Lagrangian spheres.

   Since $b_2(\Sigma) = 1$ the cohomology classes
   $[\omega_{_{\Sigma}}]$ and $[\omega_{_{\Sigma}}']$ are
   proportional, so there is a constant $c$ such that
   $[\omega_{_{\Sigma}}'] = c [\omega_{_{\Sigma}}]$. Clearly we have
   $c>0$ (to see this, take an algebraic curve $D \subset \Sigma$ and
   note that both $\int_{D} \omega_{_{\Sigma}}$ and $\int_{D}
   \omega_{_{\Sigma}}'$ must be positive since both
   $\omega_{_{\Sigma}}$ and $\omega_{_{\Sigma}}'$ are K\"{a}hler forms
   with respect to the complex structure of $\Sigma$). We claim that
   the symplectic structures $\omega_{_{\Sigma}}'$ and $c
   \omega_{_{\Sigma}}$ are diffeomorphic, i.e. there exists a
   diffeomorphism $\varphi: \Sigma \longrightarrow \Sigma$ such that
   $\varphi^* \omega_{_{\Sigma}}' = c \omega_{_{\Sigma}}$. Indeed this
   follows from Moser argument~\cite{McD-Sa:Intro} since all the forms
   in the family $\{ (1-t)c \omega_{_{\Sigma}} + t
   \omega_{_{\Sigma}}'\}_{t \in [0,1]}$ are symplectic (since
   $c\omega_{_{\Sigma}}$ and $\omega_{_{\Sigma}}'$ are both K\"{a}hler
   with respect to the same complex structure) and all lie in the same
   cohomology class.

   Since $(\Sigma, c \omega_{_{\Sigma}})$ has no Lagrangian spheres
   the same holds for $(\Sigma, \omega_{_{\Sigma}}')$ too. By
   Theorem~\ref{t:def-0-lag-sphere} we have $\tdef(X')>0$.

   That $\tdef(X') = \tdef(X)$ follows immediately from the fact that
   for manifolds with positive defect the minimal Chern number
   $C_{\Sigma}$ of a hyperplane section $\Sigma$ is determined by the
   defect. More specifically, we have (see~\S\ref{S:PDS}):
   $$\frac{n + \tdef(X)}{2} = C_{\Sigma} = 
   \frac{n + \tdef(X')}{2},$$ where $n = \dim_{\mathbb{C}}(X)$.
\end{proof}

Theorem~\ref{mt:subcrit} says that the complement of a hyperplane
section $X \setminus \Sigma$ of an algebraic manifold $X \subset
{\mathbb{C}}P^N$ with small dual is subcritical. Here is a partial
converse:

\begin{thm} \label{t:subcrit-small-dual} Let $X \subset
   {\mathbb{C}}P^N$ be an algebraic manifold with
   $n=\dim_{\mathbb{C}}(X)\geq 3$ and let $\Sigma \subset X$ be a
   hyperplane section. Assume that $(\Sigma, \omega_{_{\Sigma}})$ is
   spherically monotone with $C_{\Sigma} \geq 2$ and that
   $2C_{\Sigma}$ does not divide $n$.  If $X \setminus \Sigma$ is
   subcritical then $\tdef(X)>0$.
\end{thm}
Note that the spherical monotonicity of $(\Sigma, \omega_{_{\Sigma}})$
is automatically satisfied e.g. when $\Sigma$ is Fano and
$b_2(\Sigma)=1$.

\begin{proof}
   Suppose by contradiction that $\tdef(X)=0$. By
   Theorem~\ref{t:def-0-lag-sphere} $(\Sigma, \omega_{_{\Sigma}})$ has
   a Lagrangian sphere, say $L \subset \Sigma$. Note that since
   $\Sigma$ is spherically monotone, the Lagrangian $L \subset \Sigma$
   is monotone too and since $L$ is simply connected its minimal
   Maslov number is $N_L = 2C_{\Sigma}$.

   Put $W = X \setminus \Sigma$ endowed with the symplectic form
   $\omega_W$ induced from $X$ (which in turn is induced from
   ${\mathbb{C}}P^N$).  We now appeal to the Lagrangian circle bundle
   construction introduced in~\cite{Bi:Nonintersections,
     Bi-Ci:closed}. We briefly recall the construction. Pick a tubular
   neighborhood $\mathcal{U}$ of $\Sigma$ in $X$ whose boundary
   $\partial{\mathcal{U}}$ is a circle bundle over $\Sigma$. Denote
   this circle bundle by $\pi: \partial{\mathcal{U}} \to \Sigma$. Then
   $\Gamma_L = \pi^{-1}(L)$ is the total space of a circle bundle over
   $L$, embedded inside $W$.  By~\cite{Bi:Nonintersections}, for a
   careful choice of $\mathcal{U}$ the submanifold $\Gamma_L$ is
   Lagrangian in $W$. Moreover, since $L$ is monotone $\Gamma_L$ is
   monotone too and has the same minimal Maslov number: $N_{\Gamma_L}
   = N_L = 2C_{\Sigma}$. (See~\cite{Bi:Nonintersections} for more
   details.)

   Denote by $(\widehat{W}, \widehat{\omega}_W)$ the symplectic
   completion of the symplectic Stein manifold $(W, \omega_W)$
   (see~\cite{Bi-Ci:Stein, Bi:Nonintersections} for the details).  By
   the results of~\cite{Bi-Ci:Stein}, $\Gamma_L$ is Hamiltonianly
   displaceable (i.e. there exists a compactly supported Hamiltonian
   diffeomorphism $h: (\widehat{W}, \widehat{\omega}_W)
   \longrightarrow (\widehat{W}, \widehat{\omega}_W)$ such that
   $h(\Gamma_L) \cap \Gamma_L = \emptyset$). In particular,
   $HF(\Gamma_L, \Gamma_L)=0$.

   One can arrive now at a contradiction by using an alternative
   method to compute $HF(\Gamma_L, \Gamma_L)$ such as the Oh spectral
   sequence~\cite{Oh:spectral, Bi:Nonintersections}. (This is a
   spectral sequence whose initial page is the singular homology of
   $\Gamma_L$ and which converges to $HF(\Gamma_L, \Gamma_L)$, which
   is $0$ in our case.) We will not perform this computation here
   since the relevant part of it has already been done
   in~\cite{Bi:Nonintersections}, hence we will use the latter.

   Here are the details. We first claim that the bundle
   $\pi|_{\Gamma_L}: \Gamma_L \to L$ is topologically trivial. To see
   this denote by $N_{\Sigma /X}$ the normal bundle of $\Sigma$ in
   $X$, viewed as a complex line bundle. Note that $\Gamma_L \to L$ is
   just the circle bundle associated to $N_{\Sigma /X}|_L$. Thus it is
   enough to show that $N_{\Sigma/X}|_L$ is trivial. Denote by $c \in
   H^2(\Sigma;\mathbb{Z})$ the first Chern class of $N_{\Sigma/X}$ and
   by $c_{\mathbb{R}}$ its image in $H^2(\Sigma;\mathbb{R})$.
   Similarly, denote by $c|_L$ and by ${c_\mathbb{R}}|_L$ the
   restrictions of $c$ and $c_{\mathbb{R}}$ to $L$. As $\Sigma \subset
   X$ is a hyperplane section we have $c_{\mathbb{R}} =
   [\omega_{_{\Sigma}}]$. But $L \subset (\Sigma, \omega_{_{\Sigma}})$
   is Lagrangian hence $c_{\mathbb{R}}|_L = 0$. As $H^*(L;\mathbb{Z})$
   has no torsion ($L$ is a sphere) it follows that $c|_L = 0$ too.
   Thus the restriction $N_{\Sigma/X}|_L$ of $N_{\Sigma/X}$ to $L$ has
   zero first Chern class. This implies that the line bundle
   $N_{\Sigma/X}|_L \to L$ is trivial (as a smooth complex line
   bundle). In particular $\Gamma_L \to L$ is a trivial circle bundle.

   Since $\Gamma_L \approx L \times S^1$ we have
   $H^i(\Gamma_L;\mathbb{Z}_2) = \mathbb{Z}_2$ for $i=0,1, n-1, n$ and
   $H^i(\Gamma_L;\mathbb{Z}_2)=0$ for every other $i$. By
   Proposition~6.A of~\cite{Bi:Nonintersections} we have $2C_{\Sigma}
   \mid n$. A contradiction. (Note that the conditions $n \geq 3$ and
   $C_{\Sigma} \geq 2$ in the statement of the theorem are in fact
   required for Proposition~6.A of~\cite{Bi:Nonintersections} to
   hold.)
\end{proof}

\subsection{Other approaches to proving Corollary~\ref{mc:periodicity}}
\cntrsb \label{sb:different-prf-periodicity}

Here we briefly outline an alternative approach to proving
Corollary~\ref{mc:periodicity} and possibly Theorem~\ref{mt:om-inv},
based on the subcriticality of $X \setminus \Sigma$ that was
established in Theorem~\ref{mt:subcrit}.

Put $W = X \setminus \Sigma$ and $\omega_W$ be the symplectic form on
$W$ induced from that of $X$. Let $\mathcal{U}$ be a tubular
neighborhood of $\Sigma$ in $X$ as in the proof of
Theorem~\ref{t:subcrit-small-dual}. The boundary $P =
\partial{\mathcal{U}}$ of $\mathcal{U}$ is a circle bundle $\pi : P
\longrightarrow \Sigma$ over $\Sigma$. Consider the embedding $$i: P
\longrightarrow W \times \Sigma, \quad i(p) = (p, \pi(p)).$$ Denote by
$\Gamma_{P} = i(P) \subset W \times \Sigma$ the image of $i$.  By the
results of~\cite{Bi:Nonintersections}, one can choose $\mathcal{U}$ in
such a way that there exists a positive constant (depending on the
precise choice of $\mathcal{U}$) such that $i(P)$ is a Lagrangian
submanifold of $(W \times \Sigma, \omega_W \oplus -c
\omega_{_{\Sigma}})$. (Note the minus sign in front of
$\omega_{_{\Sigma}}$.) Moreover, the Lagrangian $\Gamma_P$ is monotone
and its minimal Maslov number is $N_P = 2C_{\Sigma}$, where
$C_{\Sigma}$ is the minimal Chern number of $\Sigma$. So by the
results recalled in~\S\ref{S:PDS} we have $N_P = n+k$. Note that
$\dim_{\mathbb{R}} \Gamma_P = 2n+1$.

As $W$ is subcritical it follows that $\Gamma_P$ can be Hamiltonianly
displaced in the completion $(\widehat{W} \times \Sigma,
\widehat{\omega}_W \oplus -c \omega_{_{\Sigma}})$ and therefore
$$HF(\Gamma_P, \Gamma_P) = 0.$$ (See~\cite{Bi:Nonintersections} for
the details. See also the proof of Theorem~\ref{t:subcrit-small-dual}
above.) Note that in order to use here Floer cohomology with ground
coefficient ring $\mathbb{Z}$ we need to have $\Gamma_P$ oriented and
endowed with a spin structure. In our case, $\Gamma_P$ carries a
natural orientation and it is easy to see that it has a spin structure
(in fact, it is easy to see that $H^1(P;\mathbb{Z}_2)=0$ hence this
spin structure is unique).

We now appeal to the Oh spectral sequence~\cite{Oh:spectral,
  Bi:Nonintersections}. Recall that this is a spectral sequence whose
first page is the singular cohomology of $\Gamma_P$ and which
converges to the Floer cohomology $HF(\Gamma_P, \Gamma_P)$. A simple
computation shows that in our case, due to the fact that $N_P = n+k$,
this sequence collapses at the second page, and moreover since
$HF(\Gamma_P, \Gamma_P)=0$ this second page is $0$ everywhere.  By
analyzing the differentials on the first page we obtain the following
exact sequences for every $j \in \mathbb{Z}$:
\begin{equation} \label{eq:spec-seq-gamma-L}
   H^{j-1+n+k}(\Gamma_P;\mathbb{Z}) \longrightarrow
   H^j(\Gamma_P;\mathbb{Z}) \longrightarrow
   H^{j+1-n-k}(\Gamma_P;\mathbb{Z}).
\end{equation}
This implies many restrictions on the cohomology of $P \approx
\Gamma_P$, e.g. that $H^j(P;\mathbb{Z}) = 0$ for every $n-k+3 \leq j
\leq n+k-2$, that $H^j(P;\mathbb{Z}) \cong H^{j-1+n+k}(P;\mathbb{Z})$
for every $0 \leq j \leq n-k-2$ and more.  We now substitute this
information into the Gysin sequences of the bundle $P \longrightarrow
\Sigma$ (whose Euler class is just the hyperplane class $h$
corresponding to the embedding $\Sigma \subset {\mathbb{C}}P^N$).
Combining the calculation via the Gysin sequences together with the
Lefschetz theorem yields the desired periodicity for the cohomology of
$\Sigma$. We omit the details as they are rather straightforward.

One could try to push the above argument further by using the methods
of~\cite{Bi-Kh:floer-gysin} (see e.g.~\S 14 in that paper) in order to
prove Theorem~\ref{mt:om-inv} via Lagrangian Floer cohomology.
However, this would require an extension of the methods
of~\cite{Bi-Kh:floer-gysin} to coefficients in $\mathbb{Z}$ rather
than just $\mathbb{Z}_2$.

\section{What happens in the non-monotone case} \cntrs
\label{s:non-monotone}
Here we briefly explain what happens in Theorem~\ref{t:om-inv-gnrl}
when the condition ``$c_1^X(A) = \lambda h(A)$ for some $\lambda >
0$'' is not satisfied, e.g. when $(\Sigma, \omega_{_{\Sigma}})$ is not
spherically monotone (see Definition~\ref{df:monotone}).

We will need to change here a bit our coefficient ring for the quantum
cohomology since $(\Sigma, \omega_{_{\Sigma}})$ is not spherically
monotone anymore. Denote by $\mathcal{A}$ the ring of all formal
series in the variables $q$, $T$
$$P(q,T) = \sum_{i,j} a_{i,j} q^i T^{s_j}, \quad a_{i,j} \in \mathbb{Z}, 
s_j \in \mathbb{R},$$ which satisfy that for every $C \in \mathbb{R}$
$$\# \bigl\{ (i,j) \mid a_{i,j} \neq 0 \; \textnormal{and } s_j>C \bigr\} 
< \infty.$$ This ring is a special case of the more general Novikov
ring commonly used in the theory of quantum cohomology. With this ring
as coefficients, the definition of the quantum product $*$ on
$QH(\Sigma;\mathcal{A})$ is very similar to what we have had before.
Namely, the powers of the variable $q$ will encode Chern numbers of
rational curves involved in the definition of $*$ and the powers of
$T$ encode their symplectic areas. See~\cite{McD-Sa:Jhol-2} for more
details.

We now turn to the Hamiltonian fibration $\pi_{\ell}: \widetilde{X}
\longrightarrow \ell$.  We will use here the construction and notation
from~\S\ref{s:small-dual-fibr} and~\S\ref{s:proof-om-inv}.
Additionally, denote by $\widetilde{i}:\Sigma \longrightarrow
\widetilde{X}$ the inclusion of the fiber into the total space of the
fibration $\pi_{\ell} : \widetilde{X} \longrightarrow \ell$. Recall
also from~\S\ref{s:proof-om-inv} that we have a canonical injection
$j: H_2(X;\mathbb{Z}) \longrightarrow H_2(\widetilde{X};\mathbb{Z})$
which satisfies $j \circ p_* = \textnormal{id}$, where $p:
\widetilde{X} \longrightarrow X$ is the blow down map. Denote by
$B_{\ell} \subset X$ the base locus of the pencil $\ell$. With this
notation we have:
\begin{equation}\label{eq:wtld-i}
   \widetilde{i}_*(\alpha) = j(\alpha) - ([B_{\ell}] \cdot \alpha)F =  
   j(\alpha) - 
   \langle [\omega_{_{\Sigma}}], \alpha \rangle F \; \; \; \forall 
   \alpha \in H_2(\Sigma;\mathbb{Z}).
\end{equation}
The symplectic form $\widetilde{\Omega}$ satisfies:
\begin{equation} \label{eq:wtld-Om-wtld-i-etc}
   \begin{aligned}
      & [\widetilde{\Omega}] = 2 p^*[\omega_{_X}] - e, \;\;
      \textnormal{ where } e \in H^2(\widetilde{X}) \textnormal{ is the
        Poincar\'{e} dual of } E,
      \\
      & \langle [\widetilde{\Omega}], j(A) \rangle = 2 \langle
      [\omega_{_X}], A \rangle \; \;\; \forall A \in H_2(X;\mathbb{Z}), \\
      & \langle [\widetilde{\Omega}], F \rangle = 1.
   \end{aligned}
\end{equation}

The Seidel element of the fibration $\pi_{\ell}:
\widetilde{X} \longrightarrow \ell$ will now be:
$$S(\pi_{\ell}) =
\sum_{\widetilde{A} \in H_2^{\pi}}
\mathcal{S}(\widetilde{A};\widetilde{J}) \otimes
q^{\nu(\widetilde{A})}T^{\langle [\widetilde{\Omega}], \widetilde{A}
  \rangle} \in QH^{-2c_0(\pi_{\ell})}(\Sigma;\mathcal{A}).$$

Some parts of the proof of Theorem~\ref{r:pencils} go through in this
new setting. More specifically, Lemma~\ref{l:curves-1} as well as
Lemma~\ref{l:sect-class-F} continue to hold (with the same proofs) and
it follows that the contribution of the class $F$ to the Seidel
element is as before, namely 
\begin{equation} \label{eq:SF} 
   \mathcal{S}(F) = [\omega_{_{\Sigma}}].
\end{equation}
If we choose as before the reference class of sections to be $F$ then
the total degree of the Seidel element $S(\pi_{\ell})$ continues to be
$2$.

In contrast to the above, Lemma~\ref{l:d<1} does not hold anymore
since we might have holomorphic sections in the class $\widetilde{A} =
j(A) + d F$ with $d \leq 0$. (We will see
in~\S\ref{sb:non-monotone-example} an example in which this is indeed
the case.) Nevertheless we can still obtain some information on
$S(\pi_{\ell})$ beyond~\eqref{eq:SF}. Let $d \in \mathbb{Z}$ and put
$\widetilde{A} = j(A) + dF$ where $A \in H_2^S(X)$. Recall from
Lemma~\ref{l:curves-1} that $\widetilde{A}$ might contribute to
$S(\pi_{\ell})$ only if the following three conditions are satisfied:
\begin{enumerate}
  \item $d\leq 1$.
  \item $[\Sigma] \cdot A = 1-d$.
  \item $A \in \mathcal{R}(X)$ where $\mathcal{R}(X) \subset H_2^S(X)$
   is the positive cone generated by those classes that can be
   represented by $J_0$-holomorphic rational curves.
   (See~\S\ref{s:proof-om-inv}.)
\end{enumerate}
Moreover, $d=1$ iff $A = 0$.

The case $d=1$ has already been treated in~\eqref{eq:SF}. Assume that
$d \leq 0$. A simple computation shows that
$$\langle [\widetilde{\Omega}], \widetilde{A} \rangle = 2-d, \quad 
\langle c_1^v, \widetilde{A} \rangle = -1 + \langle c_1^X -h, A
\rangle.$$ Here $h \in H^2(X)$ is the hyperplane class corresponding
to the embedding $X \subset {\mathbb{C}}P^N$, i.e. $h = PD([\Sigma])$.
This proves the following theorem:
\begin{thm} \label{t:gnrl-S-pi-ell} Let $X \subset {\mathbb{C}}P^N$ be
   an algebraic manifold with small dual and $\Sigma \subset X$ a
   hyperplane section. Then the Seidel element $S(\pi_{\ell})$
   corresponding to the fibration $\pi_{\ell}:\widetilde{X}
   \longrightarrow \ell$ is given by:
   \begin{equation} \label{eq:S-gnrl} S(\pi_{\ell}) =
      [\omega_{_{\Sigma}}]T + \sum_{d \leq 0, A} \mathcal{S}(j(A) +
      dF) T^{2-d} q^{(c_1^X(A)-h(A))/ C_{\Sigma}},
   \end{equation}
 where the sum is taken over all $d \leq 0$ and $A
   \in \mathcal{R}(X)$ with:
   \begin{enumerate}
     \item $h(A) = 1-d$.
     \item $3-d-n \leq c_1^X(A) \leq 2-d$.
   \end{enumerate}
   In particular, if $-K_X - \Sigma$ is nef and $\min \{(-K_X -
   \Sigma) \cdot A \mid A \in \mathcal{R}(X)\} \geq 2$ then
   $$S(\pi_{\ell}) = [\omega_{_{\Sigma}}]T.$$
\end{thm}
Note that the powers of $T$ in the second summand of~\eqref{eq:S-gnrl}
are always $\geq 2$ and the powers of $q$ in the second summand are
always $\leq 1$ (but might in general be also negative).

Here is a non-monotone example, not covered by
Theorem~\ref{t:om-inv-gnrl} but to which Theorem~\ref{t:gnrl-S-pi-ell}
does apply. Let $X = {\mathbb{C}}P^{m+r} \times {\mathbb{C}}P^m$ with
$m\geq 2$ and $r \geq 1$ be embedded in
${\mathbb{C}}P^{(m+1)(m+r+1)-1}$ by the Segre embedding. It is well
known that $\tdef(X) = r$ (see Theorem~6.5 in~\cite{Te:dual-arxiv}).
It is easy to see that $c_1^X - h$ is ample and since $m \geq 2$ its
minimal value on $\mathcal{R}(X)$ is $m \geq 2$. It follows that
$S(\pi_{\ell}) = [\omega_{_{\Sigma}}]T \in QH^2(\Sigma;\mathcal{A})$.

This calculation fails to be true when $m=1$, as will be shown
in~\S\ref{sb:non-monotone-example} below.

\subsection{A non-monotone example} \cntrsb
\label{sb:non-monotone-example}

Consider the algebraic manifold $\Sigma = \mathbb{C}P^1 \times
\mathbb{C}P^1$. Denote by $f, s \in H_2(\Sigma;\mathbb{Z})$ the
classes 
\begin{equation} \label{eq:fs}
   f = [\textnormal{pt} \times \mathbb{C}P^1], \quad s =
   [\mathbb{C}P^1 \times \textnormal{pt}].
\end{equation}
We have $H_2^S(\Sigma) = H_2(\Sigma;\mathbb{Z}) = \mathbb{Z} s \oplus
\mathbb{Z} f$. Denote by $\alpha, \beta \in H^2(\Sigma)$ the
Poincar\'{e} duals of $f$, $s$ respectively, i.e.:
\begin{equation} \label{eq:alpabetta}
   \langle \alpha, s \rangle = 1, \; \; \langle \alpha, f \rangle = 0, 
   \quad 
   \langle \beta, s \rangle = 0, \; \; \langle \beta, f \rangle = 1.
\end{equation}
A simple computation shows that $$c_1^{\Sigma} = 2\alpha + 2\beta.$$

Before we continue, a small remark about our algebro-geometric
conventions is in order. For a complex vector space $V$ we denote by
$\mathbb{P}(V)$ the space of complex lines through $0$ ({\em not} the
space of hyperplanes or $1$-dimensional quotients of $V$). Similarly,
for a vector bundle $E \to B$ we denote by $\mathbb{P}(E) \to B$ the
fiber bundle whose fiber over $x \in B$ is $\mathbb{P}(E_x)$, as just
defined, i.e. the space of lines through $0$ in $E_x$.  We denote by
$T \to \mathbb{P}(E)$ the tautological bundle, which by our
convention, is defined as the line bundle whose fiber over $l \in
\mathbb{P}(E_x)$ is the line $l$ itself. We denote by $T^*$ the dual
of $T$, i.e.  $T^*_l = \hom (l, \mathbb{C})$. For example, with these
conventions, for $E = \mathbb{C}^{n+1}$ (viewed as a bundle over $B =
\textnormal{pt}$) we have $T^* = \mathcal{O}_{{\mathbb{C}}P^n}(1)$,
and $T^*$ is ample.

Consider now the bundle $\mathcal{O}_{\mathbb{C}P^1}(-1)$ over
$\mathbb{C}P^1$. There is an obvious inclusion $$\iota:
\mathcal{O}_{\mathbb{C}P^1}(-1) \longrightarrow
\mathcal{O}_{\mathbb{C}P^1} \oplus \mathcal{O}_{\mathbb{C}P^1}$$
coming from viewing an element $l \in \mathbb{C}P^1$ as a subspace $l
\subset \mathbb{C} \oplus \mathbb{C}$.  Consider now the inclusion:
\begin{equation} \label{eq:incl-1} \mathcal{O}_{\mathbb{C}P^1}(-1)
   \oplus \mathcal{O}_{\mathbb{C}P^1}(-1) \xrightarrow{\textnormal{id}
     \oplus \iota} \mathcal{O}_{\mathbb{C}P^1}(-1) \oplus
   \mathcal{O}_{\mathbb{C}P^1} \oplus \mathcal{O}_{\mathbb{C}P^1}.
\end{equation}
Denote by $E$ the bundle on the right-hand side of this inclusion and
by $E'$ the bundle on the left-hand side. Put $$X = \mathbb{P}(E)$$
and denote by $\textnormal{pr}: X \longrightarrow \mathbb{C}P^1$ the
bundle projection.  Note that $\mathbb{P}(E') \cong
\mathbb{P}(\mathcal{O}_{\mathbb{C}P^1} \oplus
\mathcal{O}_{\mathbb{C}P^1}) = \mathbb{C}P^1 \times \mathbb{C}P^1 =
\Sigma$ hence~\eqref{eq:incl-1} induces an embedding $i_{\Sigma,X}:
\Sigma \longrightarrow X$. Let $T \to X$ be the tautological bundle
(as previously defined) and consider the bundle
$$\mathcal{L} = T^* \otimes p^* \mathcal{O}_{\mathbb{C}P^1}(1).$$ 

\begin{thm} \label{t:ex-cp1xcp1-2} The line bundle $\mathcal{L}$ is
   very ample and the projective embedding of $X$ induced by it has
   $\tdef = 1$. The embedding of $\Sigma$, $i_{\Sigma,X}(\Sigma)
   \subset X$, is a smooth hyperplane section of the projective
   embedding of $X$ induced by $\mathcal{L}$. Moreover if
   $\omega_{_X}$ is the symplectic structure on $X$ induced by the
   projective embedding of $\mathcal{L}$ and $\omega_{_{\Sigma}} =
   i_{\Sigma,X}^* \omega_{_X}$ then we have:
   $$[\omega_{_{\Sigma}}] = 2\alpha + \beta.$$ If $\ell$ is a pencil in
   the linear system $| \mathcal{L} |$ lying in the complement of the
   dual variety $X^*$ then the Seidel element of the fibration
   $\pi_{\ell}:\widetilde{X} \longrightarrow \ell$ associated to
   $\ell$ is: $$S(\pi_{\ell})= (2 \alpha + \beta) T + \beta T^2.$$
\end{thm}
The proof is given in~\S\ref{sb:prf-cp1xcp1-2} below. One can easily
generalize the above example to other projective bundles and also to
higher dimensions.

Note that $[\omega_{_{\Sigma}}]$ and $c_1^{\Sigma}$ are not
proportional hence the conditions of Theorem~\ref{t:om-inv-gnrl} are
not satisfied. It is also easy to see that (for homological reasons)
$(\Sigma, \omega_{_{\Sigma}})$ does not contain any Lagrangian spheres
(c.f. Theorems~\ref{t:def-0-lag-sphere},~\ref{t:no-lag-spheres}).

The quantum product for $(\Sigma, \omega_{_{\Sigma}})$ is given by
(see~\cite{McD-Sa:Jhol-2}):
$$\alpha*\alpha = q T^2, \quad \beta*\beta = qT, \quad 
\alpha*\beta = \alpha \cup \beta, \;\; \textnormal{where }
\deg(q)=4.$$ The inverse of $S(\pi_{\ell})$ in quantum cohomology is
given by
$$S(\pi_{\ell})^{-1} = \tfrac{1}{qT^2(1-T)^2}
\bigl( -2\alpha + (1+T)\beta \bigr).$$ Here we have written
$\tfrac{1}{(1-T)^2}$ as an abbreviation for $(\sum_{q=0}^{\infty}
T^j)^2$.

In contrast to the situation in Theorem~\ref{t:ex-cp1xcp1-2} we can
exhibit the {\em same} manifold $\Sigma = \mathbb{C}P^1 \times
\mathbb{C}P^1$ as a hyperplane section of a different projective
manifold $X'$ with $\tdef = 0$, but with different induced symplectic
structure (c.f. Theorem~\ref{c:sig-hyp-sect}). This goes as follows.
Let $X' \subset {\mathbb{C}}P^5$ be the image of the degree--$2$
Veronese embedding of ${\mathbb{C}}P^3$.  It is well known that
$\tdef(X') = 0$. A simple computation shows that a smooth hyperplane
section of $X'$ is isomorphic to $\Sigma = \mathbb{C}P^1 \times
\mathbb{C}P^1$. The symplectic form $\omega'_{_{\Sigma}}$ on $\Sigma$
induced from the standard symplectic structure of ${\mathbb{C}}P^5$
satisfies
$$[\omega'_{_{\Sigma}}] = \alpha + \beta.$$ Note that $(\Sigma,
\omega'_{_{\Sigma}})$ has Lagrangian spheres. To see that, note that
(e.g. by Moser argument) $\omega'_{_{\Sigma}}$ is diffeomorphic to the
split form $\omega_0 = \sigma \oplus \sigma$, where $\sigma$ is the
standard K\"{a}hler form on $\mathbb{C}P^1$. The symplectic manifold
$(\Sigma, \omega_0)$ obviously has Lagrangian spheres, for example $L
= \{ (z, w) \mid w = \bar{z} \},$ hence so does $(\Sigma,
\omega'_{_{\Sigma}})$. (c.f.
Theorems~\ref{t:def-0-lag-sphere},~\ref{t:no-lag-spheres}).

Finally, note that $[\omega'_{_{\Sigma}}] = \alpha + \beta$ is not
invertible in the quantum cohomology $QH(\Sigma;\Lambda)$. In fact, a
simple computation shows that $(\alpha + \beta) * (\alpha -\beta) =
0$.

\subsection{Proof of Theorem~\ref{t:ex-cp1xcp1-2}} \cntrsb
\label{sb:prf-cp1xcp1-2}
Consider the inclusion
\begin{equation} \label{eq:incl-2} \mathcal{O}_{\mathbb{C}P^1}(-1)
   \oplus \mathcal{O}_{\mathbb{C}P^1} \oplus
   \mathcal{O}_{\mathbb{C}P^1} \xrightarrow{\iota \oplus
     \textnormal{id} \oplus \textnormal{id}}
   \mathcal{O}_{\mathbb{C}P^1}^{\oplus 4}.
\end{equation}
Put $Y = \mathbb{P}(\mathcal{O}_{\mathbb{C}P^1}^{\oplus 4}) =
\mathbb{C}P^1 \times {\mathbb{C}}P^3$ and denote by
$\textnormal{pr}_1:Y \to \mathbb{C}P^1$ and $\textnormal{pr}_2: Y \to
{\mathbb{C}}P^3$ the projections. The inclusion~\eqref{eq:incl-2}
gives us an obvious inclusion $i_{X,Y}: X \longrightarrow Y$. The
bundle $T^*$ naturally extends to $Y$ as $\textnormal{pr}_2^*
\mathcal{O}_{\mathbb{C}P^3}(1)$, and
$\textnormal{pr}^*\mathcal{O}_{\mathbb{C}P^1}(1)$ extends to $Y$ too
as $\textnormal{pr}_1^* \mathcal{O}_{\mathbb{C}P^1}(1)$. It follows
that
$$\mathcal{L} = \widetilde{\mathcal{L}}|_X, \; \textnormal{where }
\widetilde{\mathcal{L}} = \textnormal{pr}_2^*
\mathcal{O}_{{\mathbb{C}}P^3}(1) \otimes
\textnormal{pr}_1^*\mathcal{O}_{\mathbb{C}P^1}(1).$$ The bundle
$\widetilde{\mathcal{L}}$ is obviously very ample, hence so is
$\mathcal{L}$. Moreover it is well known that the embedding $Y \subset
{\mathbb{C}}P^7$ induced by $\widetilde{\mathcal{L}}$ (the Segre
embedding) has $\tdef(Y) = 2$ (see e.g.~\cite{Te:dual-arxiv}).  A
straightforward computation shows that $i_{X,Y}(X) \subset Y$ is
indeed a hyperplane section corresponding to the embedding $Y \subset
{\mathbb{C}}P^7$ induced by $\widetilde{\mathcal{L}}$. If we denote by
$ [x_0:x_1]$ homogeneous coordinates on $\mathbb{C}P^1$ and by
$[w_0:w_1:w_2:w_3]$ homogeneous coordinates on ${\mathbb{C}}P^3$ then
$i_{X,Y}(X) \subset Y$ is given by the equation $$i_{X,Y}(X) =
\{x_0w_1 - x_1 w_0 = 0 \} \subset \mathbb{C}P^1 \times
{\mathbb{C}}P^3.$$ As for the defect of the projective embedding of
$X$, we have by~\eqref{eq:def-Sigma-X} that $\tdef(i_{X,Y}(X)) =
\tdef(Y) - 1 = 1$.

It remains to show that $i_{\Sigma,X}(\Sigma) \subset X$ is indeed a
hyperplane section corresponding to $\mathcal{L}$. Using the
coordinates $x, w$ just introduced, we can write a point in $X$ as
$([x_0:x_1], [w_0:w_1:w_2:w_3])$ with $(w_0, w_1) \in \mathbb{C}(x_0,
x_1)$, $w_2, w_3 \in \mathbb{C}$. A smooth hyperplane section of $X$
(with respect to $\mathcal{L}$, or alternatively with respect to the
embedding $X \subset {\mathbb{C}}P^6$) is given for example by the
equation
\begin{equation} \label{eq:Sig_0} \Sigma_0 = \{x_0w_2 - x_1 w_3 = 0\}.
\end{equation}
A simple computation shows that $\Sigma_0 = i_{\Sigma,X}(\Sigma)$.

Next we construct a pencil $\ell$ of divisors in the linear system
$|\mathcal{L}|$ lying in the complement of the dual variety $X^*$
(corresponding to the projective embedding induced by $\mathcal{L}$).
For this end, we first construct a linear embedding $$\iota' :
\mathcal{O}_{\mathbb{C}P^1}(-2) \longrightarrow
\mathcal{O}_{\mathbb{C}P^1}(-1) \oplus \mathcal{O}_{\mathbb{C}P^1}$$
as follows. Write elements of the fiber of
$\mathcal{O}_{\mathbb{C}P^1}(-1)$ over $[x_0:x_1] \in \mathbb{C}P^1$
as pairs $(v_0, v_1) \in \mathbb{C}(x_0,x_1) \subset \mathbb{C}^2$ (or
in coordinates: $v_0x_1 = v_1 x_0$). Similarly, write elements of
$\mathcal{O}_{\mathbb{C}P^1}(-2) = \mathcal{O}_{\mathbb{C}P^1}(-1)
\otimes \mathcal{O}_{\mathbb{C}P^1}(-1)$ as $\bigl( (u_0, u_1) \otimes
(u'_0, u'_1) \bigr)$ with $(u_0, u_1), (u'_0, u'_1) \in
\mathcal{O}_{\mathbb{C}P^1}(-1)$. Define
$$\iota' \bigl( (u_0, u_1) \otimes (u'_0, u'_1) \bigr) =
\bigl( (u_1 u'_0, u_1 u'_1), u_0u'_0 \bigr).$$ We get an embedding
$\mathcal{O}_{\mathbb{C}P^1}(-2) \oplus \mathcal{O}_{\mathbb{C}P^1}
\xrightarrow{\iota' \oplus \textnormal{id}}
\mathcal{O}_{\mathbb{C}P^1}(-1) \oplus \mathcal{P}_{\mathbb{C}P^1}
\oplus \mathcal{O}_{\mathbb{C}P^1}$ hence also an embedding
$$i': \mathbb{P}(\mathcal{O}_{\mathbb{C}P^1}(-2) \oplus 
\mathcal{O}_{\mathbb{C}P^1}) \longrightarrow X.$$ Put $\Sigma_1 =
i'(\mathbb{P}(\mathcal{O}_{\mathbb{C}P^1}(-2) \oplus
\mathcal{O}_{\mathbb{C}P^1}) \subset X$.

A simple computation shows that $\Sigma_1$ is given by the following
equation $$\Sigma_1 = \{ x_0 w_0 - x_1 w_2 = 0\} \subset X.$$ It
follows that $\Sigma_1$ lies in the linear system $|\mathcal{L}|$.
Consider now the pencil $\ell \subset |\mathcal{L}|$ generated by
$\Sigma_0$ (see~\eqref{eq:Sig_0}) and $\Sigma_1$.  A straightforward
computation shows that $\ell$ lies in the complement of the dual
variety $X^*$. Note that a generic element in the pencil $\ell$ is
isomorphic to $\mathbb{C}P^1 \times \mathbb{C}P^1$, however finitely
many elements in $\ell$ are isomorphic to the Hirzebruch surface $F_2
= \mathbb{P}(\mathcal{O}_{\mathbb{C}P^1}(-2) \oplus
\mathcal{O}_{\mathbb{C}P^1})$.

Denote by $s_{X} \in H_2(X;\mathbb{Z})$ the homology class represented
by the rational curve
$$\mathbb{P}(0 \oplus 0 \oplus \mathcal{O}_{\mathbb{C}P^1}) \subset X.$$
Denote by $f_X \in H_2(X;\mathbb{Z})$ the class represented by a
projective line lying in a fiber of the projective bundle $X
\longrightarrow \mathbb{C}P^1$. Clearly $(i_{\Sigma,X})_* (f) = f_X$
and $(i_{\Sigma,X})_*(s) = s_X + f_X$. (See~\eqref{eq:fs}.) The base
locus $B_{\ell}$ of the pencil $\ell$ is a smooth algebraic curve
whose homology class in $\Sigma$ is $[B_{\ell}] = s+2f$, and when
viewed in $X$ we have $$(i_{\Sigma,X})_*([B_{\ell}]) = s_X + 3f_X.$$

Denote by $\pi_{\ell}: \widetilde{X} \longrightarrow \ell$ the
fibration associated to $\ell$. We endow $\widetilde{X}$ with the
symplectic form $\widetilde{\Omega}$ as in~\S\ref{s:small-dual-fibr}.
We will now compute the Seidel element $S(\pi_{\ell})$.

Denote by $\omega_{_Y}$ the symplectic form on $Y$ induced from the
standard symplectic form on ${\mathbb{C}}P^7$ via the embedding $Y
\subset {\mathbb{C}}P^7$. Denote also $\omega_{_X} = i_{X,Y}^*
\omega_{_Y}$ and $\omega_{_{\Sigma}} = i_{\Sigma,X}^* \omega_{_X}$ the
induced forms on $X$ and $\Sigma$. A simple computation shows that
$[\omega_{_Y}] = \textnormal{pr}_1^*a + \textnormal{pr}_2^*b$, where
$a \in H^2(\mathbb{C}P^1;\mathbb{Z})$ and $b \in
H^2({\mathbb{C}}P^3;\mathbb{Z})$ are the positive generators. A
straightforward computation shows now that
$$[\omega_{_{\Sigma}}] = 2\alpha + \beta.$$

We now go back to the situation of Lemma~\ref{l:d<1} (which does not
hold in our case) and try to find the contribution of holomorphic
sections of $\pi_{\ell}: \widetilde{X} \longrightarrow \ell$ in the
class $\widetilde{A} = j(A) + d F$ with $d\leq 0$.

A simple computation shows that $j(s_X) \in H_2^{\pi_{\ell}}$ and that
$c_1^v(j(s_X)) = -1$. We also have (see~\eqref{eq:wtld-i})
$$\widetilde{i}_*(f) = j(f_X) - F, \quad 
\widetilde{i}_*(s) = j(s_X) + j(f_X) - 2F.$$ The degree of the Seidel
element $S(\pi_{\ell})$ is in our case $2$, and as $C_{\Sigma} = 2$ it
follows that the only classes $\widetilde{A}$ that might contribute to
$S(\pi_{\ell})$ are classes $\widetilde{A}$ that differ from $s_X$ by
an element coming from $\widetilde{i}_*:H_2(\Sigma;\mathbb{Z})
\longrightarrow H_2(\widetilde{X};\mathbb{Z})$ with zero
$c_1^{\Sigma}$. This means that
$$\widetilde{A} = j(s_X) + r\widetilde{i}_*(s-f) = 
(r+1)j(s_X) - rF, \; \; r \in \mathbb{Z}.$$ As $\widetilde{A}$
represents a holomorphic section we must have $0< \langle
[\widetilde{\Omega}], \widetilde{A} \rangle = r+2$, hence $r \geq -1$.

\begin{lem} \label{l:sect-r=-1-or-r=0}
   If $\mathcal{S}(\widetilde{A}) \neq 0$ then either $r=-1$ or $r=0$.
\end{lem}
We postpone the proof for later in this section and continue with the
proof of Theorem~\ref{t:ex-cp1xcp1-2}.

The case $r=-1$ is when $\widetilde{A} = F$. This has already been
treated at the beginning of~\S\ref{s:non-monotone} and we have
$\mathcal{S}(F) = [\omega_{_{\Sigma}}] = 2\alpha + \beta$.

We turn to the case $r=0$, i.e. $\widetilde{A} = j(s_X)$. Denote by
$\widetilde{J_0}$ the standard complex structure on $\widetilde{X}
\subset X \times \ell$, namely the complex structure induced from the
standard complex structure $J_0 \oplus i$ on $X \times \ell$. 

Consider for every $[u:v] \in \mathbb{C}P^1$ the $J_0$-holomorphic
rational curve $\sigma_{[u:v]} : \mathbb{C}P^1 \longrightarrow X$
defined by $\sigma_{[u:v]}([x_0:x_1]) = [0:u:v]$, where $[0:u:v]$ on
the right-hand side lies in the fiber over $[x_0:x_1]$, i.e.
$\sigma_{[u:v]}$ is a section of the projective bundle $X \to
\mathbb{C}P^1$ and $[\sigma_{[u:v]}] = s_X$.  In the family of curves
$\{ \sigma_{\xi} \}_{\xi \in \mathbb{C}P^1}$ there are precisely two
curves, say $\sigma_{\xi'}$ and $\sigma_{\xi''}$, that intersect the
base locus $B_{\ell}$ of the pencil $\ell$. (We assume here that the
pencil $\ell$ was chosen to be generic.) Thus, each of the curves
$\sigma_{\xi}$, $\xi \in \mathbb{C}P^1 \setminus \{\xi', \xi''\}$,
lifts to a holomorphic curve $\widetilde{\sigma}_{\xi}: \mathbb{C}P^1
\longrightarrow \widetilde{X}$.  (Recall that $\widetilde{X}$ is the
blow up of $X$ along the base locus $B_{\ell}$.) The homology class of
$\widetilde{\sigma}_{\xi}$ is $j(s_X)$. A simple computation shows
that each $\widetilde{\sigma}_{\xi}$, $\xi \neq \xi', \xi''$, is a
section of the fibration $\pi_{\ell} : \widetilde{X} \longrightarrow
\ell \cong \mathbb{C}P^1$ and moreover these are all the
$\widetilde{J_0}$-holomorphic sections in the class $j(s_X)$.

It is interesting to note that the moduli space $\mathcal{M}^{\frak
  s}(\widetilde{A}, \widetilde{J_0})$ of sections in the class
$\widetilde{A} = j(s_X)$ is not compact. Indeed $\mathcal{M}^{\frak
  s}(\widetilde{A}, \widetilde{J_0})$ can be identified with
$\mathbb{C}P^1 \setminus \{ \xi', \xi'' \}$ and as $\xi \to \xi'$ or
$\xi''$ we obtain bubbling in the fiber. More precisely, when $\xi \to
\xi_*$ with $\xi_* \in \{\xi', \xi''\}$ the sections
$\widetilde{\sigma}_{\xi}$ converge to a reducible curve consisting of
two components: the first is a $\widetilde{J_0}$-holomorphic section
in the class $F$ and the other one is a holomorphic curve in the class
$\widetilde{i}_*(s-f)$ lying in one of the fibers of $\pi_{\ell}$. The
latter is a $(-2)$--curve hence this can occur only in one of the
fibers that is isomorphic to the Hirzebruch surface $F_2$ (obviously
not in any of the fibers that are isomorphic to $\mathbb{C}P^1 \times
\mathbb{C}P^1$, since those are Fano).

The corresponding evaluation map $\mathcal{M}^{\frak s}(\widetilde{A},
\widetilde{J_0}) \longrightarrow \Sigma$ gives a pseudo-cycle whose
homology class is $s \in H_2(\Sigma;\mathbb{Z})$. Moreover for each
$\widetilde{\sigma}_{\xi} \in \mathcal{M}^{\frak s}(\widetilde{A},
\widetilde{J_0})$ we have $$\widetilde{\sigma}_x^* (T \widetilde{X})
\cong \mathcal{O}_{\ell}(2) \oplus \mathcal{O}_{\ell}(-1) \oplus
\mathcal{O}_{\ell}$$ hence by the regularity criterion
from~\cite{McD-Sa:Jhol-2} (see Lemma~3.3.1 in that book)
$\widetilde{J_0}$ is regular for all the elements in
$\mathcal{M}^{\frak s}(\widetilde{A}, \widetilde{J_0})$. Consequently
we have: $$\mathcal{S}(\widetilde{A}) = PD(s) = \beta.$$ Putting
everything together we see that
$$S(\pi_{\ell}) = [\omega_{_{\Sigma}}]T^{\widetilde{\Omega}(F)} + 
\beta T^{\widetilde{\Omega}(j(S_X))} = (2\alpha + \beta)T + \beta
T^2.$$ The proof of Theorem~\ref{t:ex-cp1xcp1-2} is now complete
modulo the proof of Lemma~\ref{l:sect-r=-1-or-r=0}.  \Qed

\begin{proof}[Proof of Lemma~\ref{l:sect-r=-1-or-r=0}]
   Write $\widetilde{A} = (r+1)j(s_X) - rF$. We have
   $\widetilde{\Omega}(\widetilde{A}) = r+2$. Since
   $[\widetilde{\Omega}]$ is an integral cohomology class we must have
   $r\geq -1$. Thus we have to prove that it is impossible to have
   $\mathcal{S}(\widetilde{A},\widetilde{J}) \neq 0$ for generic
   $\widetilde{J}$ if $r \geq 1$.

   \subsubsection*{Claim 1.} {\sl There exist no
     $\widetilde{J_0}$-holomorphic sections in the class
     $\widetilde{A}$ for $r\geq 1$.}

   Indeed, assume by contradiction that $\widetilde{u}_0$ is such a
   section. Recall that we have the blow down projection
   $p:\widetilde{X} \longrightarrow X$ which is $(\widetilde{J}_0,
   J_0)$-holomorphic.  Therefore $p \circ \widetilde{u}_0$ is a
   $J_0$-holomorphic curve in $X$. Its homology class is $(r+1)s_X$.
   Recall that we also have the inclusion $i_{X,Y}: X \longrightarrow
   Y = \mathbb{C}P^1 \times {\mathbb{C}}P^3$ which is holomorphic.
   Hence $u: = i_{X,Y} \circ p \circ \widetilde{u}_0$ is a holomorphic
   curve in the class $(r+1)[\mathbb{C}P^1 \times \textnormal{pt}] \in
   H_2(Y;\mathbb{Z})$. As the projection $\textnormal{pr}_2 : Y
   \longrightarrow {\mathbb{C}}P^3$ is holomorphic it follows that $u$
   is an $(r+1)$-multiply covered curve whose image is $\mathbb{C}P^1
   \times \textnormal{pt}$. Note that $\widetilde{u}_0: \mathbb{C}P^1
   \longrightarrow \widetilde{X}$ is injective (since it is a section)
   and it is easy to see (by an explicit calculation) that
   $\widetilde{u}_0$ cannot be entirely contained inside the
   exceptional divisor $E$ (of the blow up $p: \widetilde{X}
   \longrightarrow X$).  Therefore $p \circ \widetilde{u}_0$ must be
   generically injective (i.e. it restricts to an injective map over
   an open dense subset of $\mathbb{C}P^1$), and the same should hold
   also for $u$. But we have seen that $u$ is $(r+1)$-multiply covered
   curve. We thus obtain a contradiction if $r \geq 1$. This proves
   Claim~1.

   Next, consider $\widetilde{J_0}$-holomorphic {\em cusp sections} in
   the class $\widetilde{A}$. By this we mean a bubble tree of
   $\widetilde{J_0}$-holomorphic curves $(\widetilde{u}_0, v_1, \ldots
   , v_q)$ consisting of one holomorphic section $\widetilde{u}_0$
   together with holomorphic rational curves $v_i:\mathbb{C}P^1
   \longrightarrow \widetilde{X}$ each lying in a fiber of
   $\pi_{\ell}$. Moreover we have
   \begin{equation} \label{eq:cusp-sect}
      [\widetilde{u}_0] + [v_1] + \cdots + [v_q] = \widetilde{A},
   \end{equation}
   where $[\widetilde{u}_0] = (\widetilde{u}_0)_*([\mathbb{C}P^1])$
   and similarly for the $[v_i]$'s. Note however, that some of the
   $v_i$'s might not be reduced, i.e.  they might be multiply covered.
   
   \subsubsection*{Claim 2.} {\sl For $r \geq 1$ the only possible
     cusp sections $(\widetilde{u}_0, v_1, \ldots, v_q)$ in the class
     $\widetilde{A}$ are such that $\widetilde{u}_0$ is a curve in the
     class $F$ which is a fiber of the fibration $p|_E:E
     \longrightarrow B_{\ell}$ and each $v_i$ being a (possible
     multiple cover of a) $(-2)$--curve in its corresponding fiber of
     $\pi_{\ell}$. (Thus the bubbles $v_i$ may occur only in the
     special fibers of $\pi_{\ell}$ which are Hirzebruch surfaces
     different than $\mathbb{C}P^1 \times \mathbb{C}P^1$.)  Moreover,
     there exist only finitely many cusp sections in the class
     $\widetilde{A}$.}

   To prove this, let $(\widetilde{u}_0, v_1, \ldots, v_q)$ be such a
   cusp section. We first claim that $\widetilde{u}_0$ has its image
   inside the exceptional divisor $E$. Indeed, suppose otherwise, and
   consider the curve $u: = i_{X,Y} \circ p \circ \widetilde{u}_0$ as
   well as the curves $w_i: = i_{X,Y} \circ p \circ v_i$. The sum of
   the curves $u$ and $w_1, \ldots, w_q$ represent together the class
   $(r+1)j(s_X)$.  As in the proof of Claim~1 it follows that the
   union of their images is a curve of the type $(r+1)(\mathbb{C}P^1
   \times \textnormal{pt})$. By assumption $u$ is not constant hence
   $u$ is a curve of the type $m(\mathbb{C}P^1 \times
   \textnormal{pt})$, and as $\widetilde{u}_0$ is a section the curve
   $u$ must be reduced (as $p$ is $1-1$ outside of $E$). So $m=1$.
   Similarly each of the curves $w_i: = i_{X,Y} \circ p \circ v_i$
   must be either constant or with the same image as $u$, i.e.
   $\mathbb{C}P^1 \times \textnormal{pt}$ (but $w_i$ might be a
   multiple cover of $u$). As $r\geq 1$ at least one of the $w_i$'s is
   not constant.  However this is impossible since each of the curves
   $v_i$ lies in a fiber of $\pi_{\ell}$ hence the images of the
   non-constant $w_i$'s cannot coincide with that of $u$.  A
   contradiction. This proves that the image of $\widetilde{u}_0$ lies
   inside $E$.

   Next we have to prove that $\widetilde{u}_0$ is one of the fibers
   of the projective bundle $E \longrightarrow B_{\ell}$ (hence
   represents the class $F$). Indeed assume the contrary, then the
   projection $p \circ \widetilde{u}_0$ is not constant and, as
   $\widetilde{u}_0$ is contained inside $E$, $p \circ
   \widetilde{u}_0$ must be the base locus $B_{\ell}$ or a multiple
   cover of it (recall that $B_{\ell}$ is the center of the blow-up
   $p: \widetilde{X} \longrightarrow X$). Recall that $[B_{\ell}] =
   s_X + 3 f_X$. But by the preceding arguments $p \circ
   \widetilde{u}_0$ must be a curve in the class $s_X$ or a multiple
   of it.  A contradiction.  This proves that $p \circ
   \widetilde{u}_0$ is constant. It follows that $\widetilde{u}_0$ is
   either a fiber of $p|_{E}: E \longrightarrow B_{\ell}$ or a
   multiple of it.  But $\widetilde{u}_0$ is a injective (because it
   is a section of $\pi_{\ell}$), so it is precisely a fiber of $E
   \longrightarrow B_{\ell}$.

   Next we prove that each of the bubbles $v_i$ is a $(-2)$--curve or
   a multiple of it. To see this note that each $v_i$ when viewed as a
   curve in a fibre of $\pi_{\ell}$ must be in the homology class $a_i
   s - b_i f$ with $a_i, b_i \in \mathbb{Z}$ and a simple computation
   shows that
   $$\widetilde{i}_*([v_i]) = a_i j(s_X) + (a_i-b_i)j(f_X) - (2a_i -
   b_i)F.$$ This implies that $a_i \geq b_i$. But we also have
   $\sum_{i=1}^q (a_i-b_i) = 0$ because of~\eqref{eq:cusp-sect}, hence
   $a_i = b_i$ for every $i$. It follows that each $v_i$ is a curve in
   the class $a_i(s-f)$. As $\langle c_1^{\Sigma}, s-f \rangle = 0$ it
   easily follows by adjunction that $[v_i]$ is an $a_i$-cover of a
   $(-2)$--curve in the class $s-f$.

   Finally, we prove that we have only a finite number of cusp
   sections in the class $\widetilde{A}$. Inside the pencil $\ell$ we
   have only a finite number of elements $z \in \ell$ over which the
   fiber $\Sigma_{z} = \pi_{\ell}^{-1}(z)$ is not biholomorphic to
   $\mathbb{C}P^1 \times \mathbb{C}P^1$ (namely the ones that are
   Hirzebruch surfaces $F_2$). Inside each of these there is a unique
   $(-2)$-curve say $C_z$. Consider the evaluation map
   $$ev_{\widetilde{J_0},z} :
   \mathcal{M}^{\frak s}(F, \widetilde{J_0}) \longrightarrow \Sigma_z,
   \quad ev_{\widetilde{J_0},z}(\widetilde{u}) = \widetilde{u}(z).$$
   It is easy to see that the image of $ev_{\widetilde{J_0},z}$ is
   just the base locus $B_{\ell}$ (or more precisely its image inside
   the proper transform of $\Sigma_z$ inside $\widetilde{X}$). As
   $B_{\ell}$ is an irreducible curve in $\Sigma_z$ with positive self
   intersection it intersects $C_z$ at finitely many points. Therefore
   the number of cusp sections in the class $\widetilde{A}$ is finite.
   This completes the proof of Claim~2.

   We are now ready to complete the proof of the lemma. Suppose by
   contradiction that $\mathcal{S}(\widetilde{A}) \neq 0$. Consider
   the fiber $\Sigma_0$ of $\pi_{\ell}$, say lying over the point $z_*
   \in \ell$ (recall that $\Sigma_0$ is isomorphic to $\mathbb{C}P^1
   \times \mathbb{C}P^1$). Consider all possible
   $\widetilde{J}_0$-holomorphic cusp sections in the class
   $\widetilde{A}$. By Claim~2 we have only a finite number of them
   and each of them intersects $\Sigma_0$ exactly at one point (the
   bubbles cannot be inside $\Sigma_0$ as they are all
   $(-2)$--curves). We thus obtain a finite number of points $p_1,
   \ldots, p_{\nu} \in \Sigma_0$. As $\mathcal{S}(\widetilde{A}) \neq
   0$ we can find a real $2$-dimensional cycle (actually a real smooth
   closed surface) $Q \subset \Sigma_0$ lying in the complement of
   $p_1, \ldots, p_{\nu}$ and such that $\langle
   \mathcal{S}(\widetilde{A}),[Q] \rangle \neq 0$.  This implies that
   for every regular almost complex structure $\widetilde{J} \in
   \widetilde{\mathcal{J}}_{\textnormal{reg}}(\pi,
   \widetilde{\Omega})$ we have a $\widetilde{J}$-holomorphic section
   $\widetilde{u}$ in the class $\widetilde{A}$ which intersects $Q$.
   Take a sequence $\widetilde{J}_n \in
   \widetilde{\mathcal{J}}_{\textnormal{reg}}(\pi,
   \widetilde{\Omega})$ with $\widetilde{J}_n \longrightarrow
   \widetilde{J_0}$ as $n \longrightarrow \infty$ and a corresponding
   sequence $\widetilde{u}_n \in \mathcal{M}^{\frak s}(\widetilde{A},
   \widetilde{J}_n)$ with $\widetilde{u}_n(z_*) \in Q$. By Gromov
   compactness the sequence $\widetilde{u}_n$ either has a subsequence
   that converges to a genuine $\widetilde{J_0}$-holomorphic section
   in the class $\widetilde{A}$ or there is a subsequence that
   converges to a $\widetilde{J_0}$-holomorphic cusp section
   $(\widetilde{u}_0, v_1, \ldots, v_q)$ in the class $\widetilde{A}$,
   and by our construction we must have $\widetilde{u}_0(z_*) \in Q$.
   However, both cases are impossible. The first case is ruled out by
   Claim~1 and the second case is impossible since $Q$ lies in the
   complement of $p_1, \ldots, p_{\nu}$. The proof of the lemma is now
   complete.
\end{proof}

\section{Discussion and questions} \cntrs
\label{S:discussion}

Here we briefly discuss further directions of study arising from the
results of the paper.

\subsection{Questions on the symplectic topology of manifolds with
  small dual} \cntrsb
\label{sb:discuss-symp-top}

Consider the class of manifolds $\Sigma$ that appear as hyperplane
sections of manifolds $X$ with small dual, viewed as symplectic
manifolds. Does the group of Hamiltonian diffeomorphisms
$\textnormal{Ham}(\Sigma)$ of such manifolds $\Sigma$ have sepecial
properties (from the Geomeric, or algebraic viewpoints) ?  This
question seems very much related to the subcriticality of $X \setminus
\Sigma$, and results in this direction have been recently obtained by
Borman~\cite{Bor:symp-red} who found a relation between
quasi-morphisms on $\textnormal{Ham}(\Sigma)$ and quasi-morphisms on
$\textnormal{Ham}(X)$.

The structure of the fundamental group $\pi_1(\textnormal{Ham})$ of
the group of Hamiltonian diffeomorphisms of a symplectic manifold has
been the subject of many studies in symplectic topology. Still,
relatively little is known about the structure of these fundamental
groups. (e.g. the pool of known examples of symplectic manifolds with
non-simply connected $\textnormal{Ham}$ is quite limited.) It would be
interesting to ask whether manifolds with small dual and their
hyperplane sections exhibit special properties in terms of
$\pi_1(\textnormal{Ham})$ or more generally in terms of the topology
of $\textnormal{Ham}$.

Here are more concrete questions in this direction. Let $X \subset
{\mathbb{C}}P^N$ be a manifold with small dual. Denote $k = \tdef(X)$
and let $\Sigma \subset X$ be a smooth hyperplane section, endowed
with the symplectic structure $\omega_{_{\Sigma}}$ induced from
${\mathbb{C}}P^N$. Denote by $\lambda \in
\pi_1(\textnormal{Ham}(\Sigma, \omega_{_{\Sigma}})$ the non-trivial
element coming from the fibration in~\S\ref{s:small-dual-fibr} using
the recipe of~\S\ref{sbsb:seidel-ham-loops}.
\begin{itemize}
  \item What can be said about the minimal Hofer length of the loops
   in $\textnormal{Ham}(\Sigma, \omega_{_{\Sigma}})$ in the homotopy
   class $\lambda$ ? More generally, what can be said in general about
   the length spectrum of $\textnormal{Ham}(\Sigma,
   \omega_{_{\Sigma}})$ with respect to the Hofer metric ?
   Preliminary considerations seem to indicate that at least when
   $b_2(X)=1$ the positive part of the norm of $\lambda \in
   \pi_1(\textnormal{Ham}(\Sigma, \omega_{_{\Sigma}})$ satisfies
   $\nu_+(\lambda) \leq \frac{1}{\dim_{\mathbb{C}}(\Sigma)+1}$. It
   would be interesting to verify this, and more importantly to obtain
   a bound on $\nu(\lambda)$. (See~\cite{Po:ham-book, Po:K-area} for
   the definition of these norms on $\pi_1(\textnormal{Ham})$ and ways
   to calculate them.)
  \item Can the homotopy class $\lambda$ be represented by a
   Hamiltonian circle action ? Several examples of manifolds with
   small dual indicate that this might be true. In case a Hamiltonian
   circle action does exist, is it true that it can be deformed into a
   holomorphic circle action (i.e. an action of $S^1$ by
   biholomorphisms of $\Sigma$) ?
  \item In which cases is the element $\lambda$ of finite order ?
   Whenever this is the case, does the order of $\lambda$ has any
   relation to $k = \tdef(X)$ ?
  \item In case the order of $\lambda$ is infinite, what can be said
   about the value of the Calabi homomorphism $\widetilde{Cal}$ on
   $\lambda$ ? (We view here $\lambda$ as an element of the universal
   cover $\widetilde{\textnormal{Ham}}(\Sigma, \omega_{_{\Sigma}})$.)
   See~\cite{En-Po:calqm} for the definition of $\widetilde{Cal}$ etc.
\end{itemize}

Of course, one could ask the same questions also about $X$ itself
(rather than $\Sigma$). It is currently not known what are the precise
conditions insuring that an algebraic manifold $X$ with small dual can
be realized as a hyperplane section in an algebraic manifold $Y$ (of
one dimension higher).

Another question, lying at the border between symplectic topology and
algebraic geometry is the following. The main results of this paper
show that an algebraic manifold $X \subset {\mathbb{C}}P^N$ with small
dual and $b_2(X)=1$ gives rise to a distinguished {\em non-trivial}
element $\lambda \in \pi_1(\textnormal{Ham}(\Sigma))$ where $\Sigma$
is a hyperplane section of $X$. On the other hand every homotopy class
of loops $\gamma \in \pi_1(\textnormal{Ham}(\Sigma))$ gives rise to a
Hamiltonian fibration $\pi_{\gamma}: \widetilde{M}_{\gamma}
\longrightarrow S^2$ with fiber $\Sigma$. Consider now (positive as
well as negative) iterates $\gamma = \lambda^r$, $r \in \mathbb{Z}$,
of $\lambda$ and the Hamiltonian fibrations corresponding to them
$\pi_{\lambda^r}: \widetilde{M}_{\lambda^r} \longrightarrow S^2$. Do
these fibrations correspond to an embedding of $\Sigma$ as a
hyperplane section in some algebraic manifold with positive defect ?
Or more generally, do the fibrations $\pi_{\lambda^r}$ correspond to
some geometric framework involving the algebraic geometry of $\Sigma$
and its projective embeddings ?  It seems tempting to suspect that
$\lambda^2$ for example corresponds to a fibration similar to
$\pi_{\ell} : \widetilde{X} \longrightarrow \ell \approx S^2$
(see~\S\ref{s:small-dual-fibr}) but instead of taking $\ell$ to be a
line in the complement of $X^*$ one takes $\ell$ to be a degree $2$
curve in the complement of $X^*$.

Finally, here is another general question motivated by analogies to
algebraic geometry. Can the concept of manifolds with small dual be
generalized to symplectic manifolds ? Can one define a meaningful
concept of defect ? The motivation comes from the following framework.
Let $(X, \omega)$ be a closed integral symplectic manifold (integral
means that $[\omega]$ admits a lift to $H^2(X;\mathbb{Z})$). By a
theorem of Donaldson~\cite{Do:hyperplane} $X$ admits symplectic
hyperplane sections, i.e. for $k \gg 0$ there exists a symplectic
submanifold $\Sigma$ representing the Poincar\'{e} dual to
$k[\omega]$. (Moreover, the symplectic generalization of the notion of
Lefschetz pencil, exists too~\cite{Do:lefshetz}.) Suppose now that for
some such $\Sigma$ the manifold $X \setminus \Sigma$ is subcritical.
Does this imply on $\Sigma$ and $X$ results similar to what we have
obtained in this paper ? (e.g. is $[\omega|_{\Sigma}]$ invertible in
$QH(\Sigma)$ ?)  One of the difficulty in this type of questions is
that the concept of dual variety (of a projective embedding, or of a
linear system) does not exist in the realm of symplectic manifolds and
their Donaldson hyperplane sections. Note that we are not aware of
examples of pairs $(X, \Sigma)$ with $X \setminus \Sigma$ subcritical
that are not equivalent (e.g.  symplectomorphic) to algebraic pairs
$(X', \Sigma')$.

\subsection{Questions about the algebraic geometry of manifolds with
  small dual} \cntrsb \label{sb:discus-alg-geom} We have seen that for
hyperplane sections $\Sigma$ of manifolds with small dual $X \subset
{\mathbb{C}}P^N$, $[\omega_{_{\Sigma}}] \in QH^2(\Sigma;\Lambda)$ is
invertible. Is the same true for $X$, i.e. is $[\omega_X] \in QH^2(X;
\Lambda)$ an invertible element ?  The $2$-periodicity of the Betti
number of $X$ in Corollary~\ref{mc:periodicity} indicates that this
might be the case. Note that our proof of the $2$-periodicity for $X$
was based on the $2$-periodicity for $\Sigma$ (which in turn comes
from the invertibility of $[\omega_{_{\Sigma}}]$), together with some
Lefschetz-type theorems, and did not involve any quantum cohomology
considerations for $X$.

Another circle of questions has to do with
Theorem~\ref{t:gnrl-S-pi-ell}. It would be interesting to figure out
more explicitly the terms with $d \leq 0$ in
formula~\eqref{eq:S-gnrl}. This might be possible to some extent of
explicitness using Mori theory in the special case of manifolds with
small dual (see e.g.~\cite{Be-Fa-So:discr, Te:dual,
  Belt-Som:adjunction} and the references therein). In the same spirit
it would be interesting to see if there are any topological
restrictions on $\Sigma$ and $X$ coming from the invertibility of
$S(\pi_{\ell})$ in the non-monotone case. We remark that when
$(\Sigma, \omega_{_{\Sigma}})$ is not spherically monotone one should
work with a more complicated Novikov ring $\mathcal{A}$ as explained
in~\S\ref{s:non-monotone}.

Another interesting question has to do with the structure of the
quantum cohomology $QH^*(\Sigma;\Lambda)$ of hyperplane sections
$\Sigma$ of manifolds with small dual $X$. As a corollary of
Theorem~\ref{mt:om-inv} we have obtained that in the monotone case
$QH^{\ast}(\Sigma;\Lambda)$ satisfies the relation
$[\omega_{_{\Sigma}}] \ast \alpha = q$ for some $\alpha \in
H^{n+k-2}(\Sigma)$. In some examples this turns out to be the only
relation. Thus it is tempting to ask when do we have a ring
isomorphism
$$QH^{\ast}(\Sigma ; \Lambda) \cong
\frac{(H^{\bullet}(\Sigma) \otimes \Lambda)^*}{\left < \omega \ast
     \alpha = q \right >}.$$

In a similar context, it is interesting to note that the algebraic
structure of quantum cohomology of uniruled manifolds has been studied
in a recent paper of McDuff~\cite{McD:Ham_U}. In particular,
in~\cite{McD:Ham_U} McDuff proves a general existence result for
non-trivial invertible elements of the quantum cohomology of uniruled
manifolds using purely algebraic methods. One can view part of the
results in this paper as a direct computation in the case of manifolds
with positive defect.

\bibliography{bibliography}

\end{document}